\documentclass{amsart}
\usepackage[utf8]{inputenc}
\usepackage[T1]{fontenc}
\usepackage{lmodern}
\usepackage{amsmath}
\usepackage{amssymb}
\usepackage{mathtools}
\usepackage{latexsym}
\usepackage[lite]{amsrefs}
\usepackage{nicefrac}
\usepackage{microtype}
\usepackage{color}
\usepackage{tikz-cd}
\usepackage{enumitem} % advanced enumerate, which handles label and ref
\setlist[enumerate,1]{label=\textup{(\arabic*)}}% ensure enumerates in theorems are upright

\usepackage[all]{xy}
\newdir{ >}{{}*!/-5pt/@{>}}

\usepackage[pdftitle={Analytification functors and homotopy epimorphisms},
pdfauthor={Oren Ben-Bassat and Devarshi Mukherjee},
pdfsubject={Mathematics}
]{hyperref}

\BibSpec{book}{%
  +{}  {\PrintPrimary}                {transition}
  +{,} { \textit}                     {title}
  +{.} { }                            {part}
  +{:} { \textit}                     {subtitle}
  +{,} { \PrintEdition}               {edition}
  +{}  { \PrintEditorsB}              {editor}
  +{,} { \PrintTranslatorsC}          {translator}
  +{,} { \PrintContributions}         {contribution}
  +{,} { }                            {series}
  +{,} { \voltext}                    {volume}
  +{,} { }                            {publisher}
  +{,} { }                            {organization}
  +{,} { }                            {address}
  +{,} { \PrintDateB}                 {date}
  +{,} { }                            {status}
  +{}  { \parenthesize}               {language}
  +{}  { \PrintTranslation}           {translation}
  +{;} { \PrintReprint}               {reprint}
  +{.} { }                            {note}
  +{.} {}                             {transition}
  +{} { \PrintDOI}                   {doi}
  +{} { available at \url}            {eprint}
  +{}  {\SentenceSpace \PrintReviews} {review}
}

\renewcommand*{\PrintDOI}[1]{\href{http://dx.doi.org/\detokenize{#1}}{doi: \detokenize{#1}}}

\newcommand{\comment}[1]{}  %to comment out chunks of text

\theoremstyle{plain}
\newtheorem{theorem}{Theorem}[section]
\newtheorem{lemma}[theorem]{Lemma}
\newtheorem{corollary}[theorem]{Corollary}

\newtheorem{proposition}[theorem]{Proposition}
\theoremstyle{remark}
\newtheorem{remark}[theorem]{Remark}
\theoremstyle{definition}
\newtheorem{definition}[theorem]{Definition}
\newtheorem{example}[theorem]{Example}
\numberwithin{theorem}{section}

\newcommand\C{\mathbb C}
\newcommand\N{\mathbb N}
\newcommand\Q{\mathbb Q}
\newcommand\R{\mathbb R}
\newcommand\Z{\mathbb Z}

\newcommand{\comb}{\overbracket[.7pt][1.4pt]}% bornological completion
\newcommand*{\Fil}{\mathcal F}% filtration

\newcommand{\defeq}{\mathrel{:=}} % per Definition

\newcommand*{\into}{\rightarrowtail}
\newcommand*{\onto}{\twoheadrightarrow}

\newcommand*{\ling}[1]{#1_\mathrm{lg}}% linear growth bornology
\newcommand{\wotimes}{\widehat{\otimes}}
\newcommand{\ootimes}{\overline{\otimes}}
\DeclarePairedDelimiter{\abs}{\lvert}{\rvert}% absolute value
\DeclarePairedDelimiter{\norm}{\lVert}{\rVert}% norm
\DeclarePairedDelimiter{\gen}{\langle}{\rangle}% generated by
\DeclarePairedDelimiter{\ceil}{\lceil}{\rceil}% round up
\DeclarePairedDelimiterX{\setgiven}[2]{\{}{\}}{#1\,{:}\,\mathopen{}#2}% set given by

\DeclareMathOperator{\coker}{coker}

\DeclareMathOperator{\Spec}{Spec}% Spektrum

\DeclareMathOperator{\colim}{colim}
\DeclareMathOperator{\coim}{coim}
\DeclareMathOperator{\im}{im}

\newcommand{\sslash}{\mathbin{/\mkern-6mu/}}

\newcommand{\fC}{\mathsf{C}}
\newcommand{\fM}{\mathsf{M}}
\newcommand{\Ind}{\mathsf{Ind}}
\newcommand{\fD}{\mathsf{D}}

\newcommand{\op}{\mathrm{op}}% opposite algebra
\newcommand{\id}{\mathrm{id}}% identity map
\newcommand{\dvr}{V}% discrete valuation ring
\newcommand{\dvgen}{\pi}% uniformiser
\newcommand{\dvf}{F}% field of fractions of \dvr
\DeclareMathOperator{\Hom}{Hom}% space of linear maps
\DeclareMathOperator{\HP}{HP}% periodic cyclic homology
\DeclareMathOperator{\HH}{HH}% Hochschild homology

\begin{document}
\title{Analytification, localization and homotopy epimorphisms}
\ \

\author{Oren Ben-Bassat}

\address{Oren Ben-Bassat\\ Department of Mathematics\\
  University of Haifa\\
  Haifa, 3498838\\
  Israel}

\email{ben-bassat@math.haifa.ac.il}

\author{Devarshi Mukherjee}

\address{Devarshi Mukherjee \\ Mathematisches Institut\\
  Georg-August Universit\"at G\"ottingen\\
  Bunsenstra\ss{}e 3--5\\
  37073 G\"ottingen\\
  Germany}

\email{devarshi.mukherjee@mathematik.uni-goettingen.de}

\begin{abstract}
We study the interaction between various analytification functors, and a class of morphisms of rings, called homotopy epimorphisms. An analytification functor assigns to a simplicial commutative algebra over a ring R, along with a choice of Banach structure on R, a commutative monoid in the monoidal model category of simplicial ind-Banach R-modules.  We show that several analytifications relevant to analytic geometry - such as Tate, overconvergent, Stein analytification, and formal completion - are homotopy epimorphisms. Another class of examples of homotopy epimorphisms arises from Weierstrass, Laurent and rational localizations in derived analytic geometry. As applications of this result, we prove that Hochschild homology and the cotangent complex are computable for analytic rings, and the computation relies only on known computations of Hochschild homology for polynomial rings. We show that in various senses, Hochschild homology as we define it commutes with localizations, analytifications and completions. 
\end{abstract}

\thanks{The authors would like to thank Ralf Meyer for interesting discussions in the course of writing this article.}

\maketitle

\tableofcontents

\section{Introduction}

The paradigm of this article is forthcoming work on derived analytic geometry, based on homotopical algebraic geometry by To\"en and Vezzosi. This builds on a series of articles starting with \cite{Ben-Bassat-Kremnizer:Nonarchimedean_analytic}, in which the authors develop various versions of analytic geometry relative to symmetric monoidal categories. Briefly,  given a closed, symmetric monoidal category \(\mathsf{C}\), one can define \textit{affine schemes} abstractly as objects in the opposite category \(\mathsf{Aff}(\fC) \defeq \mathsf{Comm}(\fC)^\op\) of unital, commutative monoids over \(\mathsf{C}\). If \(A \in \mathsf{Comm}(\fC)\) is a commutative monoid, we denote by \(\mathsf{Spec}(A)\) the corresponding object in the opposite category \(\mathsf{Aff}(\fC)\). The category of affine schemes is then equipped with a suitable collection of covers that turns it into a Grothendieck site. By choosing \(\fC = \mathsf{Mod}_\Z\) to be the category of abelian groups and duals \((\Spec(B_i) \to \Spec(A))_i\) of flat morphisms \(A \to B_i\) satisfying \(B_i \otimes_A B_i \cong B_i\) as covers, one can recover the Zariski topology on \(\mathsf{Aff}(\fC)\). A \textit{scheme} is then a set-valued sheaf on the category \(\mathsf{Aff}(C)\) for the Zariski topology, obtained by gluing affine schemes along these covers.

In analytic geometry, one can no longer reasonably use flat epimorphisms as covers for a Grothendieck topology. Indeed, if one works in the category of Fr\'echet spaces  - which is a large enough category containing algebras of complex analytic functions - then for an open subset \(U\) of a Stein space \((X,\mathcal{O}_X)\), \(\mathcal{O}_X(U)\) is usually not topologically flat as a Fr\'echet \(\mathcal{O}_X(X)\)-module (see \cite{aristov2020open}). Furthermore, the category of Fr\'echet spaces is not closed due to a plethora of topologies one can impose on the mapping space between two Fr\'echet spaces. Instead, the authors in \cite{Ben-Bassat-Kremnizer:Nonarchimedean_analytic} propose that the correct category to do relative geometry is the category \(\mathsf{Ind}(\mathsf{Ban}_R)\) of inductive systems of Banach \(R\)-modules over a Banach ring \(R\). Concretely, they equip the subcategory \(\mathsf{Comm}(\mathsf{Ban}_R) \subseteq \mathsf{Comm}(\mathsf{Ind}(\mathsf{Ban}_R))\) with the topology whose covers are defined by duals of \textit{homotopy epimorphisms} with certain properties. These are bounded \(R\)-algebra homomorphisms \(A \to B\) that satisfy \(B \wotimes_A^{\mathbb{L}} B \cong A\), where \(\wotimes_A^{\mathbb{L}}\) is the completed, projective derived tensor product. This yields a Grothendieck topology on \(\mathsf{Comm}(\mathsf{Ind}(\mathsf{Ban}_R))\), which restricts on the category of (dagger) affinoid algebras, and Stein algebras to the weak \(G\)-topology from dagger and Berkovich geometry, and the Stein topology from complex geometry.

This article is organised as follows.

In Section \ref{sec:preliminaries}, we introduce the categorical background in which we do functional analysis. Briefly, we need a generalisation of abelian categories to combine homological algebra and functional analysis. These are called \textit{quasi-abelian categories} due to Schneiders \cite{schneiders1999quasi}, and our main category of interest, namely the category \(\mathsf{Ind}(\mathsf{Ban}_R)\) of inductive systems of Banach modules over a Banach ring \(R\) is a quasi-abelian category. Section \ref{sec:homotopy_epimorphisms} proves basic inheritence properties of homotopy epimorphisms.  In Section \ref{sec:An} we introduce analytic rings of functions $C_n$ over a Banach ring $R$, so that $C_n \wotimes_{R} C_m \cong C_{m+n}$ along with injective bounded maps \begin{equation}\label{mono}R[x_1, \dots, x_n] \to C_n \to R[[x_1, \dots, x_n]].
\end{equation}

The main technical results, namely, Lemmas \ref{lem:poly_diag_strict}, \ref{lem:formal_poly_diag_strict}, \ref{lem:TateDiagStrict}, \ref{lem:daggerdaigstrict}, and \ref{lem:HoloStrict} all have a similar nature. The rings $C_n$ have underlying Ind-Ban $R$-modules which are all flat over $R$ in the sense of Definition \ref{def:flat}. In order to prove that obvious morphisms between different $C_n \to C'_n$ are homotopy epimorphisms, we will need to check that certain complexes  
\[0 \to C_2 \to C_2 \to C_1 \to 0 \]
are strictly exact meaning the first morphism is a strict monomorphism and the second is a strict epimorphism. This sequence writes diagonal functions as one variable in terms of functions of two variables modulo functions vanishing on the diagonal. We call these \textit{diagonal sequences}. In terms of non-commutative geometry in Kontsevich's formulation, these are projective resolutions of $C_1$ as a bimodule over itself showing that $C_1$ is smooth as a differential graded ring in the category of Ind-Ban $R$-modules. In all these cases, the bounded maps $\pi: C_2 \to C_1$ given by $\pi(f)(x)=f(x,x)$ are strict epimorphisms, indeed defining $\iota: C_1 \to C_2$ by $\iota(f)(y,z)=f(y)$, we have that $\pi \circ \iota = \id$, and it only remains to check that $\iota$ is ``bounded" which is obvious in all the examples. The main work is to show that first map given by multiplication by $y-z$, which preserves $C_2 \subset R[[y,z]]$ and is clearly bounded, is in fact strict. The importance of the lemmas which check this strictness condition is Theorem \ref{thm:main_1} which uses these conditions as a criteria for a map to be a homotopy epimorphism. For each of the examples of analytifications we consider except for the Tate algebras, it does not matter if $R$ is non-archimedean or not, and they are all flat as Ind-Banach $R$-modules in the sense of Definition \ref{def:flat} with respect to $\wotimes_{R}$ or $\wotimes^{na}_{R}$. Whenever we consider the ring structure on the Tate algebras, we always assume that $R$ is non-archimedean and then Tate agebras are flat as non-archimedean ind-Banach $R$-modules in the sense of Definition \ref{def:flat} with respect to $\wotimes^{na}_{R}$. For arbitrary Banach rings $R$ we sometimes do use the Tate algebras over $R$ even when $R$ is not non-archimedean but in this case please note that we are treating the Tate algebra only as an ind-Banach $R$-module and not as a ring. In Theorem \ref{thm:polynomial_dagger_isocohomological2} we prove that the maps in Equation \ref{mono} and others are homotopy epimorphisms.

Away from the `discrete' setting, the second motivation to study homotopy epimorphisms comes from derived analytic geometry. The approach adopted here uses HAG contexts, appearing in \cite{toen2008homotopical}. Briefly, a HAG context \(\mathsf{M}\) is a monoidal model category whose model structure induces one on the category of commutative monoids in \(\mathsf{M}\). The relevant model category that will be used in \cite{Ben-Bassat-Kelly-Kremnizer:Perspective} is the category \(\mathsf{s}\mathsf{Ind}(\mathsf{Ban}_R)\) with a certain projective model structure. To build derived analytic spaces, one imposes a homotopy Grothendieck topology on the opposite category of commutative monoids on \(\mathsf{M}\) and proceeds precisely as in the underived situation. Here again, the covers for the topology involve homotopy epimorphisms for simplicial commutative algebra objects in \(\mathsf{Ind}(\mathsf{Ban}_R)\). In Section \ref{sec:derived_localisation} we define, for each type of analytification $C_n$ of the functions on $n$-dimensional affine space, three types of derived localization: Weierstrass \ref{def:W}, Laurent \ref{defn:Laurent}, and rational \ref{defn:rational}.  We prove in Section \ref{sec:derived_localisation} that derived localizations are homotopy epimorphisms. In usual (underived) rigid geomertry, these are precisely the covers that lead to rigid analytic spaces.  

The next question entails the passage from the category of simplicial commutative \(R\)-algebras to \(\mathsf{Comm}(\mathsf{s}\mathsf{Ind}(\mathsf{Ban}_R))\). There are several ways this can be done, depending on the Banach structure on $R$ and on how precisely we choose to restrict the algebraic affine space \(\mathbb{A}_{R}^n = \Spec(R[t_1,\dots,t_n])\) to domains of convergence of analytic functions. We call each of these \textit{analytification functors}; they play a role similar to complex analytification (resp. rigid analytification) in complex (resp. non-archimedan) geometry. While we defer most of the properties of these functors to \cite{Ben-Bassat-Kelly-Kremnizer:Perspective}, we prove in Section \ref{sec:Analytification_hepi} that each of the analytifications we consider is a homotopy epimorphism.

Our final application arises from Hochschild homology computations. In general, the Hochschild homology of an arbitrary topological algebra is very hard to compute.  Here again, homotopy epimorphisms provide an interesting and large class of examples, whereby the Hochschild homology of analytifications can be computed from the Hochschild homology of polynomial-type algebras. This strategy has previous been observed in non-commutative geometry by Meyer (\cite{meyer2004embeddings}) and Connes (\cite{connes1985non}), who compute the Hochschild homologies of group convolution algebras and the algebra of smooth functions on the noncommutative \(2\)-torus, using a computable dense subalgebra that is `isocohomologically' embedded. Similar observations have arisen in the work of Taylor (\cite{taylor1972general}) and Pirkovskii (\cite{aristov2020open}). The main observation of Section \ref{sec:Hochschild} is that if \(A \to B\) is a homotopy epimorphism of simplicial commutative algebras in \(\mathsf{Ind}(\mathsf{Ban}_R)\), then \(\mathbb{HH}(B) \cong B \wotimes_{A}^{\mathbb{L}} \mathbb{HH}(A)\). This makes the Hochschild homology of a (derived) analytic ring \(B\) computable, when \(A\) is, for instance, the (derived) quotient of a polynomial ring. There is vast literature available in the computation of Hochschild homology for purely algebraic objects, so our results allow this to be extended to the analytic realm.

\section{Preliminaries}\label{sec:preliminaries}

\subsection{Homological algebra and functional analysis}

In this section, we recall some background material on \textit{quasi-abelian categories} developed by Schneiders \cite{schneiders1999quasi}. We require this generality in order to do homological algebra in functional analytic categories, which are seldom abelian. 

\begin{definition}\label{def:strict}
Let \(\fC\) be a category with kernels and cokernels. A morphism $f:V \to W$ is called \textit{strict} if the natural morphism from $\coker(\ker(f) \to V)$ to $\ker(W \to \coker(f))$ is an isomorphism. Denote by \(\coim(f) = \coker(\ker(f) \to V)\), and by \(\im(f) = \ker(\coker(f) \to V)\).
\end{definition}

\begin{definition}\label{def:quasi-abelian}
Let \(\fC\) be an additive category with kernels and cokernels. Then \(\fC\) is said to be \textit{quasi-abelian} if 

\begin{itemize}
\item the pushout of a strict monomorphism by an arbitrary morphism in \(\fC\) is a strict monomorphism;
\item the pullback of a strict epimorphism by an arbitrary morphism in \(\fC\) is a strict epimorphism.
\end{itemize}
 
\end{definition}

\begin{example}
Every morphism in an abelian category is strict. Consequently, every abelian category is quasi-abelian. In particular, for any ring \(R\), the category \(\mathsf{Mod}_R\) of left or right \(R\)-modules is an abelian category.  
\end{example}

\begin{example}
Let \(\mathsf{Fr}_\C\) be the category of Fr\'echet spaces over \(\C\). By the Closed Graph Theorem, a continuous linear map is a strict monomorphism if and only if it is injective and has closed range.  Dually, again by the closed graph theorem, a morphism is a strict epimorphism if and only if it is surjective. It is easy to check that strict monomorphisms (resp. strict epimorphisms) are stable under pushouts (resp. pullbacks). Finally, the category \(\mathsf{Ban}_{\C}\) of Banach spaces over \(\C\) is a full subcategory of \(\mathsf{Fr}_\C\) that is closed under extensions (that is, kernel-cokernel pairs). Therefore, \(\mathsf{Ban}_{\C}\) is a quasi-abelian category as well. 
\end{example}

There are several other examples of quasi-abelian categories that are relevant to analytic geometry, generalising the example above. But we defer a discussion about this to the next subsection \ref{subsection:borno_geo}. We will actually need a category that contains all filtered colimits - rather than just finite colimits as in the category \(\mathsf{Ban}_{\C}\). The universal way of doing this is as follows:

\begin{definition}\label{def:ind-completion} 
Let $\fC$ be a category. An \textit{ind-completion} of $\fC$ is a category $\fD$ with a functor 
$i:\fC \to \fD$, such that $D$ is closed under filtered colimits, and the functor $i$ is initial with respect to functors into categories closed under filtered colimits.
\end{definition}

\begin{lemma}\label{lem:ind-completion_exists}
Let $\fC$ be a category. Its ind-completion exists and can be realized as the full subcategory of the category $\Pr(\fC)=\mathsf{Fun}(\fC^{op},\mathsf{Set})$, whose objects are filtered colimits of representable functors (note that the category of presheaves is cocomplete). 
\end{lemma}

We will denote the ind-completion of $\fC$ by $\Ind(\fC)$.  
Given two presentations of objects $E \cong \underset{i\in I}\colim E_i$ and $F \cong \underset{j \in J}\colim F_j$, we have a canonical isomorphism
\[\Hom(E, F) \cong
\lim_{i\in I}\underset{j \in J}\colim\Hom(E_i,F_j).
\]

\begin{remark}\label{rem:DescribeHoms}
A morphism can be represented by a functor  $\alpha: I \to J$ and for each $i \in I$ an element of $\Hom(E_{i}, F_{\alpha(i)})$ giving a natural transformation $E \to F\circ \alpha$.
\end{remark}

The following is Proposition 2.1.19 in \cite{schneiders1999quasi}:

\begin{proposition}\label{prop:ScnMain}
Let $\fC$ be a small, closed, symmetric monoidal, quasi-abelian category. The category $\Ind(\fC)$ has a canonical closed symmetric monoidal structure extending that on $\fC$. Hence, if $\fC$ has enough projectives, $\Ind(\fC)$ is a closed symmetric monoidal elementary quasi-abelian category.
\end{proposition}

\begin{proof}
The tensor product and internal-Hom are simply given by: 
\begin{equation*}\label{eqn:IndRules}
\underset{i\in I}\colim E_i\ootimes \underset{j\in J}\colim F_j= \underset{(i,j) \in I\times J}\colim E_i\ootimes F_j
\end{equation*}
\begin{equation*}
\underline{\Hom}(\underset{i\in I}\colim E_i, \underset{j \in J}\colim F_j)=
\lim_{i\in I} \underset{j \in J}\colim \underline{\Hom}(E_i,F_j).
\end{equation*}
This and the rest of the claims are in \cite{schneiders1999quasi}*{Proposition 2.1.19}. 
\end{proof}

%\begin{lemma}\label{lem:ind_ban_quasi-ab}
%Let \(\fC\) be a closed symmetric monoidal quasi-abelian category. Then \(\Ind(\fC)\) is also a closed symmetric monoidal quasi-abelian category. 
%\end{lemma}

%\begin{proof}
%The inheritance of symmetric monoidal structure and the adjunction between the tensor product of \(\Ind(\fC)\) and its internal \(\Hom\) is in \cite{Meyer:HLHA}*{Chapter 1}. That the ind-completion remains quasi-abelian follows from the fact that any strict monomorphism and strict epimorphism in \(\Ind(\fC)\) can be represented by a diagram of strict epimorphisms and strict monomorphisms in \(\fC\). 
%\end{proof}

The following lemma is straightforward and the sufficient condition holds in all the quasi-abelian categories of interest to us:

\begin{lemma}\label{lem:filtered_strict}
Suppose \(\fC\) is a cocomplete quasi-abelian category in which filtered colimits commute with kernels. Then a filtered colimit of strict morphisms is strict.
\end{lemma}

\begin{proof}
Let \(f = \underset{k \in K}\colim f_k \) in \(\Ind(\fC)\), where each \(f_k \colon A_k \to B_k\) is a strict morphism, and \(K\) is a filtered category. Then \(\coim(f_k) \cong \im(f_k)\) for each \(k\), by definition. By hypothesis, we have \(\colim(\ker(B_k \to \coker(f_k))) \cong \ker(B \to \coker(f))\). Using that \(\coker\) commutes with colimits, and again that \(K\) is filtered, we get \[\colim (\coker(\ker(f_k) \to A_k)) \cong \coker(\ker(f) \to A).\] Consequently, \[\colim (\coim (f)) \cong \colim(\coim(f_k)) \cong \colim(\im(f_k)) \cong \colim(\im(f)),\]
which is what we wanted to show. 
\end{proof}

We now turn to \textit{flat} and \textit{projective} objects in a quasi-abelian category.

\begin{definition}\label{def:projective}
\begin{itemize}
\item An object $P$ in a quasi-abelian category \(\fC\) is called \textit{projective} if for any strict epimorphism $V \to W$ in \(\fC\), the induced map $\Hom(P,V)\to \Hom(P,W)$ in the category of abelian groups is surjective;
\item A quasi-abelian category \(\fC\) has \textit{enough projectives} if for any object \(X \in \fC\), there is a strict epimorphism \(P \onto X\), where \(P\) is a projective object.  
\end{itemize}
\end{definition}

\begin{definition}\label{def:flat}
An object \(X \in \fC\) in a closed symmetric monoidal quasi-abelian category with finite limits and colimits with monoidal structure $\ootimes$ is called \textit{flat} if the functor \(V \mapsto X \ootimes V\) commutes with finite limits (or preserves strict monomorphisms). A morphism \(f \colon A \to B\) in \(\mathsf{Comm}(\fC)\) is called \textit{flat} if \(B\) is flat as an object of \(\mathsf{Mod}_A\).  
\end{definition}

%\begin{remark}
%If on the other hand, we use formal colimits, then the above statement is trivial and even stronger. That is, if \(\fC\) is quasi-abelian, a morphism in \(\mathsf{Ind}(\fC)\) is strict if and only if it is a diagram of strict morphisms in \(\fC\).
%\end{remark}

\begin{lemma}
For any  \(A \in  \mathsf{Comm}(\mathsf{Ind}(\mathsf{Ban}_R))\), the quasi-abelian category \(\mathsf{Mod}(A)\) has enough projectives and any projective is flat.
\end{lemma}

\subsection{Algebraic geometry relative to symmetric monoidal quasi-abelian categories}\label{subsection:borno_geo}

The purpose of this section is to put the class of morphisms that are the subject of this article - that is, the so-called \textit{homotopy epimorphisms} - in a broader context of relative algebraic geometry. The starting point is a closed symmetric monoidal quasi-abelian category $\fC$  with monoidal structure $\ootimes$, and internal \(\Hom\) denoted by \(\underline{\Hom}_\fC\). Given this data, we can define the category of \textit{affine schemes} \(\mathsf{Aff}(\fC) \defeq \mathsf{Comm}(\fC)^\op\), where the right hand side denotes the opposite category of commutative monoids in \(\fC\). 

In order to build global objects, such as \textit{schemes} (or \textit{stacks}) relative to $\fC$, we need to equip the category of affine schemes \(\mathsf{Aff}(\fC)\) with a suitable Grothendieck topology \(\tau\). This is motivated by the familiar \textit{functor of points approach}, which first builds the category of presheaves \(\mathsf{PSh}(\fC) \defeq \mathsf{Fun}(\mathsf{Comm}(\fC), \mathsf{Set})\). The topology is then used to define the full subcategory \(\mathsf{Sh}(\fC, \tau)\) of sheaves on the site \((\fC, \tau)\) - these are functors \(F \colon \mathsf{Comm}(\fC) \to \mathsf{Set}\) satisfying \[F(X) \cong \mathsf{eq}(\prod_i F(X_i) \rightrightarrows \prod_{i,j} F(X_i \times_X X_j)),\] for each object \(X \in \mathsf{Comm}(\fC)\) and every cover \((X_i \to X)_i\) with respect to \(\tau\).  By gluing along epimorphisms of representable sheaves \(\Hom(-, X_i) \onto F\), with \(X_i \in \mathsf{Comm}(\fC)\), one obtains the category \(\mathsf{Sch}(\fC, \tau)\) of schemes relative to the site \((\fC, \tau)\). 

There is an important Grothendieck pretopology used in \cite{Ben-Bassat-Kremnizer:Nonarchimedean_analytic} - the \textit{homotopy monomorphism} topology (sometimes called the formal homotopy Zariski topology). The following morphisms appear in the covers for these topologies:

\begin{definition}\label{def:homotopy_epi}
Let \(\fC\) be a closed, symmetric monoidal quasi-abelian category. A morphism \(f \colon A \to B\) of (unital)  commutative monoids in $\fC$ is called a \textit{homotopy epimorphism} if the induced functor \(\fD(B) \to \fD(A)\) is fully faithful. The corresponding morphism in \(\mathsf{Aff}(\fC)\) is called a \textit{homotopy monomorphism}.
\end{definition}  

In order to be able the work with the definition above, we use the following reformulation, which applies to all categories of interest to us:

\begin{lemma}
Let \(f \colon A \to B\) be a morphism in \(\mathsf{Comm}(\fC)\). Assume that the induced base change functor \(\mathsf{Mod}_A \to \mathsf{Mod}_B\), \(M \mapsto B \ootimes_A M\) has a left derived functor. Then \(f\) is a homotopy epimorphism if and only if \(B \ootimes_A^{\mathbb{L}} B \cong B\) in \(\fD(B)\).  
\end{lemma}

The actual definitions of these topologies involve taking families of homotopy epimorphisms \((A \to B_i)_i\) in a homotopy transversal full subcategory \(\mathsf{A} \subseteq \mathsf{Comm}(\fC)\), that induce finite conservative subfamilies on \((\mathsf{Mod}_{\mathsf{A}}^{RR}(A) \to \mathsf{Mod}_{\mathsf{A}}^{RR}(B_i))_i\). We direct the reader to \cite{Ben-Bassat-Kremnizer:Nonarchimedean_analytic} for the relevant definitions, which we do not rewrite here as they will play no direct role in this article.  

We now recall some definitions, using which we will define the categories that are relevant to analytic geometry and also the rest of this paper. 

\begin{definition}\label{def:Banach_ring}
A \textit{Banach ring} is a commutative, unital ring \(R\) with a function \(\abs{\cdot} \colon R \to \R_{\geq 0}\) satisfying the following properties:

\begin{itemize}
\item \(\abs{a} = 0\) if and only if \(a = 0\);
\item \(\abs{a+b} \leq \abs{a} + \abs{b}\) for all \(a\), \(b \in R\);
\item there exists a \(C>0\) such that \(\abs{ab} \leq C\abs{a}\abs{b}\) for all \(a\), \(b \in R\);
\item \(R\) is complete with respect to the metric induced by the function \(\abs{\cdot}\).
\end{itemize}

\end{definition}

Examples of Banach rings include the fields of complex numbers \(\C\), the real numbers \(\R\) and the integers \(\Z\) with the usual absolute value. Other examples include the \(p\)-adic numbers \(\Q_p\) and the \(p\)-adic integers \(\Z_p\) with their \(p\)-adic norm. Finally, a degenerate example is any ring or field with the trivial norm \(\norm{x} = 1\) for all \(x \neq 0\) and \(\norm{0} = 0\). 

\begin{definition}\label{def:Ban_module}
Let \((R,\abs{\cdot}_R)\) be a Banach ring. A \textit{semi-normed \(R\)-module} is an \(R\)-module \(M\), together with a function \(\abs{\cdot}_M \colon M \to \R_{\geq 0}\) satisfying 

\begin{itemize}
\item \(\abs{0}_M = 0\);
\item \(\abs{m+n}_M \leq \abs{m}_M + \abs{n}_M\);
\item \(\abs{a\cdot m}_M \leq C \abs{a}_R \abs{m}_M\) for some \(C>0\);
\end{itemize}

A \textit{normed \(R\)-module} is a semi-normed \(R\)-module where the norm is non-degenerate, that is, \(\abs{m}_M = 0\) if and only if \(m = 0\). We call a normed \(R\)-module a \textit{Banach \(R\)-module} if it is complete with respect to the metric induced by the norm.
\end{definition}

We denote by \(\mathsf{Ban}_R\) the category of Banach \(R\)-modules. This is a symmetric monoidal category with the \textit{completed projective tensor product}, $\wotimes_{R}$ defined as the completion of the algebraic tensor product \(V \otimes_R W\) with respect to the norm \(\norm{x}_{V \otimes W} \defeq \inf \setgiven{\sum_{i \in I} \norm{m_i}\norm{n_i}}{x = \sum_{i \in I} m_i \otimes n_i, \abs{I} < \infty}\). The category \(\mathsf{Ban}_R\) also has a canonical internal-\(\Hom\), namely the \(R\)-module of bounded linear maps \(\Hom(V,W)\) with the norm  \[\norm{T} \defeq \sup_{\norm{v}\neq 0}\frac{\norm{T(v)}_W}{\norm{v}_V},\] 
turning it into a closed monoidal category. When $R$ is non-archimedean it is sometimes helpful to work with the category  \(\mathsf{Ban}^{na}_R\) of non-archimedean Banach modules, for which the  \textit{completed projective tensor product}, $\wotimes^{na}_{R}$ defined as the completion of the algebraic tensor product \(V \otimes_R W\) with respect to the norm \[\norm{x}_{V \otimes W} \defeq \inf \setgiven{\sup_{i \in I} \norm{m_i}\norm{n_i}}{x = \sum_{i \in I} m_i \otimes n_i, \abs{I} < \infty}. \]

\begin{lemma}
For any Banach ring \(R\), the category \(\mathsf{Ban}_R\) is a quasi-abelian category with enough flat projectives. Consequently, the category \(\mathsf{Ind}(\mathsf{Ban}_R)\) is also an elementary quasi-abelian category. 
\end{lemma} 

\begin{proof}
See \cite{ben2020fr} for all the claims.
\end{proof}

We recall here some details from  \cite{Ben-Bassat-Kremnizer:Nonarchimedean_analytic} and \cite{ben2020fr} on infinite products in the various categories of interest. We first need to borrow the following definition from \cite{ben2020fr}
\begin{definition}\label{defn:PsiUpsilon} Given a set $I$, let $\Psi(I)$ be the poset consisting of functions $\psi:I \to \mathbb{Z}_{\geq 1}$ with the order $\psi_{1}\leq \psi_{2}$ if $\psi_{1}(i)\leq \psi_{2}(i)$ for all $i \in I$. 
%Let $\Upsilon$ be the poset consisting of functions $\psi:I \to \mathbb{Z}_{\geq 1}$ with the order $\psi_{1}< \psi_{2}$ if $\psi_{1}(i)< \psi_{2}(i)$ for all $i \in I-J$ where $J$ is a finite subset of $I$. 
The set of objects in the category $\Psi(I)$ with is $\underset{i \in I}\prod\mathbb{Z}_{\geq 1}$.
\end{definition}
It is easy to show that for $R$ a non-archimedean Banach ring that the product of the collection of objects $V_i$ of $R$-Banach modules in the category $\mathsf{Ban}^{\leq 1}_R$ is given by
\[\{
(v_i)_{i \in I} \in \times_{i \in I}V_{i} \ \ | \ \ \sup_{i \in I} \|v_{i}\| < \infty \}  
\] equipped with the norm 
\[\|(v_i)_{i \in I} \| =\sup_{i \in I} \|v_{i}\|.
\]
For $R$ a non-archimedean Banach ring that the product of the collection of objects $V_i$ of non-archimedean Banach modules the product in the category $\mathsf{Ban}^{\leq 1, na}_R$ is given by
\[\{
(v_i)_{i \in I} \in \times_{i \in I}V_{i} \ \ | \ \ \sup_{i \in I} \|v_{i}\| < \infty \}  
\] equipped with the norm 
\[\|(v_i)_{i \in I} \| =\sup_{i \in I} \|v_{i}\|.
\]
Because these formulas exactly the same, we observe that for $R$ a non-archimedean Banach ring that the product of a collection of objects $V_i$ of non-archimedean Banach modules is the same regardless of if it is computed in $\mathsf{Ban}^{\leq 1}_R$ or $\mathsf{Ban}^{\leq 1, na}_R$. 
The product of a countable collection of objects in the category \(\Ind(\mathsf{Ban}_R)\)is given by
\[``\underset{\psi \in \Psi(I)}\colim" \prod_{i \in I}{}^{\leq 1} ((V_{i})_{\psi(i)^{-1}})
\]
whereas the product of a countable collection of objects in the category \(\Ind(\mathsf{Ban}^{na}_R)\) is given by
\[``\underset{\psi \in \Psi(I)}\colim" \prod_{i \in I}{}^{\leq 1, na} ((V_{i})_{\psi(i)^{-1}}).
\]
However, by the previous conclusion, these formulae are also the same allowing us to conclude that for $R$ a non-archimedean Banach ring and a countable collection of objects $V_i$ of non-archimedean Banach modules that the functor \(\Ind(\mathsf{Ban}^{na}_R) \to \Ind(\mathsf{Ban}_R)\) preserves these products. Notice that these are simply increasing unions since when $\psi_1 \leq \psi_2$ we have \[\prod_{i \in I}{}^{\leq 1} ((V_{i})_{\psi_1(i)^{-1}}) \subseteq \prod_{i \in I}{}^{\leq 1} ((V_{i})_{\psi_2(i)^{-1}}) \]
and similarly for the non-archimedean case.

It is shown in \cite{Ben-Bassat-Kremnizer:Nonarchimedean_analytic} that by using the machinery that we have recalled in this subsection to define schemes relative to a homotopy transversal subcategory \(\mathsf{A} \subseteq \mathsf{Comm}(\mathsf{Ban}_{\Q_p})\), one can recover rigid analytic geometry. The homotopy transversal subcategory in these cases is the category of affinoid \(\Q_p\)-algebras. For similar results in the overconvergent and Stein cases, we direct the reader to \cites{Bambozzi-Ben-Bassat-Kremnizer:Stein, Bambozzi-Ben-Bassat:Dagger}.

In this article, we will primarily work in the category \(\mathsf{Ind}(\mathsf{Ban}_R)\), although there are some occasions where the full subcategory \(\mathsf{CBorn}_R \defeq \Ind^m(\mathsf{Ban}_R)\) of essentially monomorphic objects is called into action. We mainly do this when \(R = k\) is a non-trivially valued Banach field, in which case this category has an intrinsic description in terms of \textit{bornologies}.

\begin{definition}\label{def:bornologies}
A \textit{bornology} on a set \(X\) is a collection \(\mathfrak{B}\) of its subsets, satisfying the following:

\begin{itemize}
\item for every \(x \in X\), \(\{x\} \in \mathfrak{B}\);
\item if \(S\), \(T \in \mathfrak{B}\), then \(S \cup T \in \mathfrak{B}\);
\item if \(S \in \mathfrak{B}\), and \(T \subseteq S\), then \(T \in \mathfrak{B}\). 
\end{itemize}
\end{definition}

We call members of a bornology \textit{bounded sets}. A function \(f \colon X \to Y\) between bornological sets is called \textit{bounded} if it maps bounded subsets of \(X\) to bounded subsets of \(Y\).

\begin{definition}\label{def:bornological_module}
Let \(R\) be a non-trivially valued Banach field. A \textit{bornological \(R\)-vector space} \(M\) is an \(R\)-vector space with a bornology such that the scalar multiplication and addition maps
$$R\times M\rightarrow M$$
$$M\times M\rightarrow M$$
are bounded functions.
\end{definition}

\begin{example}\label{ex:fine_bornology}
The \textit{fine bornology} on a \(k\)-vector space \(M\) is the bornology that is cofinally generated by finite-dimensional \(k\)-vector subspaces of \(M\). That is, a subset \(B \subseteq M\) is bounded if and only if there is a finite-dimensional subspace \(S \cong k^n \subseteq M\), such that \(B \subseteq S\) and \(B\) is bounded in the usual sense. 
\end{example}

The fine bornology is essentially the only natural bornology in purely algebraic situations. It turns, for instance, the polynomial algebra into a complete bornological algebra. For larger algebras such as those which arise in analytic geometry, one must use a more reasonable bornology that accounts for the analytical structure.

\begin{definition}\label{def:von-Neumann}
Let \(V\) be a locally convex topological \(k\)-vector space. The \textit{von Neumann bornology} is defined as the bornology whose bounded subsets are those which are absorbed by all neighborhoods of the origin. 
\end{definition}

To relate inductive systems with bornologies, we first note that equipping a semi-normed $k$-module with the von Neumann bornology defines a canonical functor
\[b:\mathsf{NMod}_k^{1/2}\rightarrow\mathsf{Born}_{k}^{1/2}\] from the category of semi-normed \(k\)-vector spaces to the category of bornological \(k\)-vector spaces. 

\begin{proposition}\cite{bambozzi}*{Proposition 2.1.5}\label{prop:bornologies_ind_systems}
Let \(k\) be a non-trivially valued Banach field. Then the above functor extends to a fully faithful functor \(\mathsf{Ind}^m(\mathsf{NMod}_k^{1/2})\rightarrow\mathsf{Born}_k^{1/2}\) from the category of inductive systems of semi-normed \(k\)-vector spaces to the category of bornological \(k\)-vector spaces. Its essential image consists of bornological \(k\)-vector spaces \(M\) such that for any bounded subset \(S \subseteq M\), its absolute convex hull remains bounded. The same claims hold for the categories \(\mathsf{Ind}^m(\mathsf{NMod}_k)\) and \(\mathsf{Ind}^m(\mathsf{Ban}_k)\), with the categories of separated and complete bornological \(k\)-vector spaces \(\mathsf{Born}_k\) and \(\mathsf{CBorn}_k\) on the right hand sides, respectively.
\end{proposition}

Since we are mainly interested in complete objects such as Banach algebras and inductive systems of such algebras, the category most relevant to us is the category of complete bornological \(k\)-vector spaces. It has the same homological properties as the category \(\Ind(\mathsf{Ban}_k)\), with the advantage of being a concrete category. These properties are summarized as follows:

\begin{proposition}\cite{ben2020fr}
The category of complete bornological \(k\)-vector spaces \(\mathsf{CBorn}_k\) is a closed, symmetric monoidal quasi-abelian category with enough flat projectives. 
\end{proposition}

\section{Properties of homotopy epimorphisms}\label{sec:homotopy_epimorphisms}

%\textcolor{red}{Here we should state all the important properties of homotopy epimorphisms. This includes being closed under composition, base change, etc. Quite a few of these properties hold in any symmetric monoidal quasi-abelian category where the base change functor has a derivation.}

In this section, we state some basic inheritance properties of homotopy epimorphisms, which will be used throughout this article. A homotopy epimorphism (resp. epimorphism) in an $\infty$-category (resp. old-fashioned category) is a morphism $X\to Y$ so that for any object $Z$, the natural morphism $Map(Y,Z) \to Map(X,Z)$ has all homotopy fibers empty or contractable (resp. $\Hom(Y, Z) \to \Hom(X,Z)$ is injective). Throughout this section, $\fM$ is a symmetric monoidal $\infty$-category or a monoidal model category with monoidal structure given as $\ootimes^{\mathbb{L}}$. Using finite homotopy pushouts, which we assume to be given by $\ootimes^{\mathbb{L}}$, homotopy epimorphisms in \(\mathsf{Comm}(\fM)\) are equivalently described as morphisms $A \to B$ for which the natural map $B\ootimes^{\mathbb{L}}_A B \to B$ is an equivalence. 
%\begin{lemma}\label{lem:composition}
%Let \(f \colon A \to B\) and \(g \colon B \to C\) be homotopy epimorphisms in \(\mathsf{Comm}(\fC)\). Then \(g \circ f \colon A \to C\) is a homotopy epimorphism. 
%\end{lemma} 

%\begin{proof}
%The composition of fully faithful functors is fully faithful. 
%\end{proof}

\begin{lemma}\label{lem:base_change_1}
Let \(f \colon A \to B\) be any morphism in \(\mathsf{Comm}(\fM)\). Then \(f\) is a homotopy epimorphism if and only if for any morphism \(g \colon B \to C\) in \(\mathsf{Comm}(\fM)\), we have \(B \ootimes_A^{\mathbb{L}} C \cong C\). 
\end{lemma}

\begin{proof}
We have \[B \ootimes_A^{\mathbb{L}} C \cong B \ootimes_A^{\mathbb{L}} B \ootimes_B^{\mathbb{L}} C \cong B \ootimes_B^{\mathbb{L}} C \cong C.\] The converse is trivial. 
\end{proof} 

\begin{proposition}\label{prop:two_out_of_three}
Let \(f \colon A \to B\) and \(g \colon B \to C\) be morphisms in \(\mathsf{Comm}(\fM)\). If two out of \(f\), \(g\) and \(g \circ f\) are homotopy epimorphisms, then so is the third.
\end{proposition}

\begin{proof}
Suppose \(g\) is a homotopy epimorphism, then \(\fD(B) \to \fD(C)\) is fully faithful. Therefore, the functor \(\fD(A) \to \fD(B)\) is fully faithful if and only if the functor \(\fD(A) \to \fD(C)\) induced by \(g \circ f\) is fully faithful.  Conversely, if \(f\) and \(g \circ f\) are homotopy epimorphisms, then 
\[ C \cong C \ootimes_A^{\mathbb{L}} C \cong C \ootimes_B^{\mathbb{L}} B \ootimes_A^{\mathbb{L}} C \cong C \ootimes_B^{\mathbb{L}} C,\]
where the last isomorphism follows from Lemma \ref{lem:base_change_1}. 
\end{proof}

\begin{proposition}\label{prop:tensor_isocohomological}
Let \(f \colon A \to B\), \(g \colon C \to D\),   \(h \colon S \to T\),be homotopy epimorphisms in \(\mathsf{Comm}(\fM)\).  Then the induced morphism \(f \ootimes_{h}^{\mathbb{L}} g \colon A \ootimes_{S}^{\mathbb{L}} C \to B \ootimes_{T}^{\mathbb{L}} D\) is a homotopy epimorphism.
\end{proposition}

\begin{proof}
We have \[ (B \ootimes_{T}^{\mathbb{L}}  D) \ootimes^{\mathbb{L}}_{A \ootimes_{S}^{\mathbb{L}}C} (B \ootimes_{T}^{\mathbb{L}}  D)   \cong (B \ootimes_A^{\mathbb{L}} B) \ootimes_{T \ootimes_{S}^{\mathbb{L}}  T}^{\mathbb{L}} (D \ootimes_C^{\mathbb{L}} D) \cong B \ootimes_{T}^{\mathbb{L}} D.\]
\end{proof}

\begin{proposition}\label{prop:derived_base_change}
Let \(f \colon A \to B\) be a homotopy epimorphism in \(\mathsf{Comm}(\fM)\) and let \(g \colon A \to C\) be an arbitrary morphism in  \(\mathsf{Comm}(\fM)\). Then the morphism \(C \to C \ootimes_A^{\mathbb{L}} B\) induced by the derived pushout is a homotopy epimorphism.
\end{proposition}

\begin{proof}
We have \[(C \ootimes_A^{\mathbb{L}} B) \ootimes^{\mathbb{L}}_C (C \ootimes_A^{\mathbb{L}} B) \cong C \ootimes_A^{\mathbb{L}} B \ootimes_A^{\mathbb{L}} B \cong C \ootimes_A^{\mathbb{L}} B.\]
\end{proof}

Let \(\fC\) be a closed symmetric monoidal quasi-abelian category with enough flat projectives and strictly exact filtered colimits.  Equip the simplicial objects $s\fC$ with the projective \ref{def:projective} model structure and consider $\fC \subset s\fC$ and $\mathsf{Comm}(\fC) \subset \mathsf{Comm}(s\fC)$ in the usual way. Define a homotopy epimorphism in $\mathsf{Comm}(\fC)$ to be given by the same condition $B\ootimes^{\mathbb{L}}_A B \cong B$ in the derived category. Taking the cokernel in degree zero gives the condition $B\ootimes_A B \cong B$ which is equivalent to the condition that $A \to B$ is an epimorphism in $\mathsf{Comm}(\fC)$. This can also be seen by passing to connected components of mapping spaces in the description of homotopy epimorphisms in the beginning of this section. 
\begin{proposition}\label{prop:filtered_colimit_isocohomological}
Let \(\fC\) be a closed symmetric monoidal quasi-abelian category with enough flat projectives and strictly exact filtered colimits. Then the 
filtered colimit of homotopy epimorphisms in \(\mathsf{Comm}(\fC)\) is a homotopy epimorphism in  \(\mathsf{Comm}(\fC)\).
\end{proposition}

\begin{proof}
Let \((f_i \colon A_i \to B_i)_{i \in I}\) be a filtered family of homotopy epimorphisms. Then \(B_i \ootimes_{A_i}^{\mathbb{L}} B_i \cong B_i\). Now use that taking colimits commutes over the monoidal structure, and preserves chain homotopy equivalences. 
\end{proof}

\begin{proposition}\label{prop:tensor_isocohomological_disc}
Let \(\fC\) be a closed symmetric monoidal quasi-abelian category with enough flat projectives and let \(f \colon A \to B\), \(g \colon C \to D\) be homotopy epimorphisms in \(\mathsf{Comm}(\fC)\). Assume that $A, B, C, D$ are all flat over $\id_{\fC}$. Then the induced morphism \(f \ootimes g \colon A \ootimes C \to B \ootimes D\) is a homotopy epimorphism.
\end{proposition}
\begin{proof}This is immediate from Proposition \ref{prop:tensor_isocohomological} upon taking $\fM$ to be the category of simplicial objects of $\fC$ with the projective model structure  \cite{Ben-Bassat-Kremnizer:Nonarchimedean_analytic}.
\end{proof}

\section{Analytification and strictness of diagonal sequences}\label{sec:An}

This section defines \textit{analytifications} of polynomial algebras in the category \(\fC = \mathsf{Ind}(\mathsf{Ban}_R)\). These describe the fundamental objects in analytic geometry such as disc algebras; which can in turn be obtained from free, polynomial type algebras from algebraic geometry. Schemes over $R$ in algebraic geometry can be built out of \textit{polynomial rings} \(R[x_1,\dots,x_n]\), where \(R\) is a commutative unital ring. Geometrically, this entails imposing relations (through vanishing loci of polynomials) on the \textit{affine \(n\)-space} \(\mathbb{A}_R^n \defeq \Spec(R[x_1,\dots,x_n])\). Now to get analytic spaces, we start with a commutative unital Banach ring that dictates the analysis on the resulting algebraic objects (that is, free commutative algebras and their quotients). Geometrically, the \textit{analytic affine spaces} correspond to the domain of convergence of formal power series over \(R\). Actually, different choices of domains give rise to different types of analytifications. For instance, when \(R = \Q_p\) with its \(p\)-adic norm, the Tate algebra encodes formal power series that are convergent on a polydisc of some radius. These are the affine spaces in rigid analytic geometry. On the other hand, in overconvergent geometry, one takes power series that are convergent on polydiscs of larger radii, allowed to depend on the power series. Our main result is that several different types of analytifications \(\mathbb{A}_R^{\mathrm{an},n} \hookrightarrow \mathbb{A}_R^n\) are homotopy monomorphisms and similarly, that the analytification of a closed subvariety of $\mathbb{A}_R^n$ as it sits in the closed subvariety itself is a homotopy monomorphism. 
\subsection{Symmetric and contracting symmetric algebras}\label{subsec:symmetric_algebra}

Let \(\fC\) be a closed, symmetric monoidal category with countable co-products. The forgetful functor \(\mathsf{Comm}(\fC) \to \fC\) has a left adjoint: \[S \colon \fC \to \mathsf{Comm}(\fC), \quad V \mapsto \mathsf{Sym}(V) = \amalg_{ n\geq 0} V^{\ootimes n} /S_n,\] which assigns to a module object its \textit{symmetric algebra}. The category we will mostly work in is \(\fC = \mathsf{Ind}(\mathsf{Ban}_R)\) for a Banach ring \(R\), though we sometimes also work in the concrete, full subcategory \(\mathsf{CBorn}_R\) of complete bornological \(R\)-modules. We first explicitly describe the symmetric algebra functor \(S_R \colon \Ind(\mathsf{Ban}_R) \to \mathsf{Comm}(\Ind(\mathsf{Ban}_R))\) on free and finite Banach modules, as the resulting algebras will be the first ones that we will analytify. 

Suppose \(V\) is a free Banach \(R\)-module of finite rank \(n\), then \(V = R^n\) as a Banach \(R\)-module. We can either view \(V\) as an object of \(\mathsf{Ind}(\mathsf{Ban}_R)\) via the canonical inclusion \(\mathsf{Ban}_R \hookrightarrow \Ind(\mathsf{Ban}_R)\), or we can view it as a complete bornological \(R\)-module. In terms of bornologies, since \(V\) is finite-dimensional, there is a bornological \(R\)-module isomorphism between \(V\) in the fine bornology and the von Neumann bornology, turning it into a complete bornological \(R\)-module.  Consequently, the symmetric algebra \[S_R(V) \cong R[x_1,\dots, x_n]\] when viewed as a bornological algebra with the fine bornology, is a complete bornological \(R\)-algebra. In fact, the fine bornology on \(S_R(V)\) is the unique bornology implementing the adjunction isomorphism \[\underline{\Hom}(S_R(V), A) \cong \underline{\Hom}(V,A),\] where \(A\) is an arbitrary commutative algebra object in \(\mathsf{CBorn}_R\).% or \(\mathsf{Ind}(\mathsf{Ban}_R)\).

We now realise the symmetric algebra \(S_R(V)\) as an inductive system of Banach \(R\)-modules. To do this, we need to see what the fine bornology on the polynomial algebra \(R[x_1,\dots, x_n]\) corresponds to when we view it as an object of \(\Ind(\mathsf{Ban}_R)\). Concretely, the fine bornology on the polynomial algebra is given as follows: a subset is bounded if and only if it is contained in the \(R\)-module \(\Fil_m(R[x_1,\dots,x_n])\) generated by monomials of total degree at most \(m\), that is, terms of the form \(x^\alpha = x_1^{\alpha_1}\ldots x_1^{\alpha_n}\) with \(\abs{\alpha} = \sum_{i=1}^n \abs{\alpha_i} \leq m\). It is easy to see that \((\Fil_m(R[x_1,\dots, x_n]))_{m \in \N}\) is an increasing filtration of \(R[x_1,\dots, x_n]\) by finitely generated \(R\)-submodules. Consequently, \[
R[x_1,\dots,x_n] = \varinjlim_m \Fil_m(R[x_1,\dots,x_n]) = \bigcup_m \Fil_m(R[x_1,\dots,x_n]).\]
In terms of inductive systems, \(S_R(V)\) can be viewed as a formal colimit of the finite-dimensional Banach \(R\)-modules \(\Fil_m(R[x_1,\dots,x_n])\). Since this is a strict inductive system (that is, the structure maps are inclusions), there is a bornological isomorphism between \(S_R(V)\) with the fine bornology, and the inductive limit bornology on the inductive system \((\Fil_m(R[x_1,\cdots,x_n]))_{m \in \N}\). Notice that we have defined $S_R(V)$ as an object of $\mathsf{Comm}(\Ind(\mathsf{Ban}_R))$ for any Banach ring $R$, and when $R$ is non-archimedean, the same $S_R(V)$ is an object of $\mathsf{Comm}(\Ind(\mathsf{Ban}^{na}_R))$.

It is easy to see that 
\begin{equation}\label{eqn:decomPoly}
R[x_1, \dots, x_n] \wotimes_{R} R[y_1, \dots, y_m]  \cong R[x_1, \dots, x_n, y_1, \dots, y_m]  
\end{equation}
and when $R$ is non-archimedean,
\begin{equation}\label{eqn:decomPolyNA}
R[x_1, \dots, x_n] \wotimes^{na}_{R} R[y_1, \dots, y_m]  \cong R[x_1, \dots, x_n, y_1, \dots, y_m].
\end{equation}

%We call the filtered colimit \((\Fil_m(R[x_1,\cdots, x_n])_{m \in \N}\) the \textit{fine ind-structure} on the polynomial ring. 

%But since we prefer to work in the category \(\mathsf{Ind}(\mathsf{Ban}_R)\), we will view \(S_R(V)\) as an inductive system of Banach \(R\)-modules. And, since its inductive limit is the fine bornology, we will call this formal colimit the \textit{fine ind-structure} on the polynomial ring. \textbf{If this technical remark was not clear, I can clear it up once we write more on the background material, etc.} 

%Since the category \(\mathsf{CBor}_R \hookrightarrow \mathsf{Ind}(\mathfrak{Ban}_R)\), we can also view the symmetric algebra \(S_R(V)\) as an inductive system of Banach algebras. The relevance of doing this will become clearer when we talk about Tate, dagger and Stein algebras. 

\begin{definition}[Rescalling a Banach ring]
For any real number $r > 0$ we define $R_r$ as the trivial $R$-module whose underlying abelian group agrees with that of $R$ and is promoted to a Banach $R$-module by defining a complete norm on it by the formula $\norm{x}_r \defeq r \abs{x}_{R}$.
\end{definition}

\begin{lemma}\label{lem:rescaling}
In the notation above, if $V=R_{r_1} \oplus \cdots \oplus R_{r_n}$, then $S_R(V) \cong S_{R}(R^{n})$.
\end{lemma}

\begin{proof}
Let \(r>0\) be arbitrary. The identity map on \(R\) induces an isomorphism \(R \cong R_r\) of Banach \(R\)-modules as \(\norm{x}_R= \frac{1}{r} \norm{x}_r\). Consequently, \(V \cong R^n\) as a Banach \(R\)-module. Now use the functoriality of the symmetric algebra construction.
\end{proof}

%We now talk about \textit{contracting symmetric algebras} \(S_R(V)^{\leq 1}\), which are defined as the free 
%objects in the category \(\mathsf{Comm}(\mathfrak{Ban}^{\leq 1})\) of commutative algebras over contracting
% Banach spaces. If \(V\) is free of finite rank \(n\), the contracting symmetric algebra has an explicit description
 % given by the subalgebra \[R \{x_1, \cdots, x_n\} \defeq \setgiven{\underset{J}\sum a_J x^J \in R[[x_1, \dots, x_n]]}{ \underset{J} \sum |a_J| < \infty}.\] 
 % This is a Banach algebra with norm \(\norm{\underset{J}\sum a_J x^J} \defeq \underset{J} \sum |a_J|\). 

We use the rescalled versions of the Banach ring \(R\) to define the following versions of polydisc algebras:

\begin{definition}\label{defn:Tate}
For any Banach ring $R$ and $n$-tuple of positive real numbers $r=(r_1, \dots, r_n)$ the \textit{polydisc algebra of polyradius $r$} is defined by the subring
\[R\{\frac{x_1}{r_1}, \dots, \frac{x_n}{r_n} \}=\{\underset{J}\sum a_J x^J \in R[[x_1, \dots, x_n]]\ \ | \ \ \underset{J}\sum |a_J| r^J< \infty\}
\]
equipped with the norm $\norm{\underset{J}\sum a_J x^J} \defeq \underset{J}\sum |a_J| r^J$. This is a Banach \(R\)-algebra. 
\end{definition}
This is the symmetric algebra construction in the category $\mathsf{Ban}^{\leq 1}_{R}$ applied to the object of $\mathsf{Ban}_{R}$ given by $\underset{i=1, \dots, n}\coprod R_{r_i}=R_{r_1} \oplus \cdots \oplus R_{r_n}$.
The definition above works in the archimedean as well as non-archimedean context. 
It satisfies 
\begin{equation}\label{eqn:decomDisk}
 R \{ \frac{x_1}{r_1}, \dots, \frac{x_n}{r_n} \} \wotimes_{R} R \{\frac{y_1}{s_1}, \dots, \frac{y_n}{s_n} \} \cong R \{\frac{x_1}{r_1}, \dots, \frac{x_n}{r_n},  \frac{y_1}{s_1}, \dots, \frac{y_m}{s_m}\}.
 \end{equation}
If on the other hand one wants to work exclusively in the non-archimedean context, and only takes $R$ to be a non-archimedean, then one can read this article replacing everywhere the Tate algebra in place of the above polydisc algebra. The reader can also read the article as is, since the above construction can also be used when $R$ is non-archimedean.

\begin{definition}
Let \(R\) be a Banach ring and \(r>0\) as above. The consider the Banach $R$-module

\[R \gen{\frac{x_1}{r_1}, \dots, \frac{x_n}{r_n} } \defeq \{\underset{J}\sum a_J x^J \in R[[x_1, \dots, x_n]]\ \ | \ \ \underset{J}\lim |a_J| r^J=0\}
\]
equipped with the norm $\norm{\underset{J}\sum a_J x^J}=\underset{J}\sup |a_J| r^J$. 
\end{definition}
When $R$ is non-archimedean, this is the symmetric ring construction in the category $\mathsf{Ban}^{\leq 1, na}_{R}$ applied to the object given by $\underset{i=1, \dots, n}\coprod R_{r_i}=R_{r_1} \oplus \cdots \oplus R_{r_n}$. When $R$ is archimedean, we still consider and use this object, but cannot expect it to have any universal properties.  When $R$ is non-archimedean we call this Banach ring the \textit{Tate algebra of polyradius \(r>0\)}. It satisfies 
\begin{equation}\label{eqn:decomTate}
R \gen{\frac{x_1}{r_1}, \dots, \frac{x_n}{r_n} } \wotimes^{na}_{R} R \gen{\frac{y_1}{s_1}, \dots, \frac{y_n}{s_n} } \cong R \gen{\frac{x_1}{r_1}, \dots, \frac{x_n}{r_n},  \frac{y_1}{s_1}, \dots, \frac{y_m}{s_m}}.
\end{equation}

We can similarly define Banach abelian groups $M\{\frac{x_1}{r_1}, \dots, \frac{x_n}{r_n} \}$ and $M\langle\frac{x_1}{r_1}, \dots, \frac{x_n}{r_n} \rangle$ for any  Banach abelian group $M$. Notice that these are completions of the normed groups $M[\frac{x_1}{r_1}, \dots, \frac{x_n}{r_n}]$ and they are subgroups of $M[[x_1, \dots, x_n]]$.  The following is easy to check.
\begin{lemma}\label{lem:compare}Let $R$ be any Banach ring. The map 
\[R \{ \frac{x_1}{\rho_1}, \dots,  \frac{x_n}{\rho_n}\} \to R \langle \frac{x_1}{\rho_1}, \dots,  \frac{x_n}{\rho_n}\rangle
\]
taking each power series to itself
is well defined and bounded and for every $\rho'<\rho$ we similarly have a bounded inclusion 
\[ R \langle \frac{x_1}{\rho_1}, \dots,  \frac{x_n}{\rho_n}\rangle \to R \{ \frac{x_1}{\rho'_1}, \dots,  \frac{x_n}{\rho'_n}\}.
\]
\end{lemma}

%Let \(R\) be a Banach ring and \(V\) a finite rank, free Banach \(R\)-module. Let \(S_R(V)\) be the symmetric algebra over \(V\). This is a free object in \(\mathsf{Comm}(\mathsf{Ind}(\mathsf{Ban}_R))\). As a bornological \(R\)-algebra, this is the polynomial algebra in as many variables as the rank of \(V\). The bornology is described as follows. For $V=R_{r_1} \oplus \cdots \oplus R_{r_n}$, we get the polynomials where subset of $R[x_1, \dots, x_n]$ is bounded if it is a bounded subset of some direct sum \(S_R(V)_d = \underset{\text{monomials of degree} \leq d} \bigoplus R \)  where $||f||=\underset{0\leq j\leq d}\sum \underset{|J|=j}\sum |a_{J}| r^j$. This is known as the fine bornology. Make sure its complete.

\begin{lemma}\label{lem:poly_diag_strict}
Let \(R\) be a Banach ring. Then the sequence 
\[0 \to S_{R}(R \oplus R)  \to S_{R}(R \oplus R)  \to S_{R}(R) \to 0
\]
is strictly exact in the quasi-abelian category \(\mathsf{Ind}(\mathsf{Ban}_R)\). If $R$ is non-archimedean then the same is true in  \(\mathsf{Ind}(\mathsf{Ban}^{na}_R)\). 
\end{lemma}

\begin{proof}

By Lemma \ref{lem:rescaling}, the above short exact sequence is equivalent to the short exact sequence

\[0 \to R[y,z]  \stackrel{y-z}\to R[y,z]  \to R[x] \to 0
\]
which we need to show is strict. 
\noindent Here we view the polynomial algebra as an object of \(\mathsf{Ind}(\mathsf{Ban}_R)\), using the fine ind-structure mentioned previously. We need to show that the map \[(y-z)R[y,z] \ni f \mapsto \frac{f(y,z) - f(z,z)}{y-z} \in R[y,z]\] is bounded in the fine ind-structure. In other words, we need to find a constant $C$ depending on $d$ such that for any polynomial $f(y,z)$ of degree less than or equal to $d$ which satisfies $f(z,z)=0$, we have $\norm{f(y,z)/(y-z)} \leq C \norm{f}$. Let $f(y,z) = \sum_{i,j} a_{i,j}y^iz^j$, so that the norm $\norm{f} = \sum_{i,j} |a_{i,j}|$ of \(f\) is a finite sum of the norms of its coefficients. Now

\[\frac{f(y,z)-f(z,z)}{y-z} = \sum_{i>0 , j \geq 0} \sum_{t=0}^{i-1}a_{i,j} y^{i-1-t}z^{j+t}=  \sum_{k \geq 0 , l \geq 0}   (\sum^l_{t=0} a_{k+1+t,l-t}) y^k z^l
\]
has norm at most 
\begin{multline*}
\sum_{k \geq 0 , l \geq 0}   |(\sum^l_{t=0} a_{k+1+t,l-t})|\leq \sum_{k \geq 0 , l \geq 0}(l +1)\max_{t=0}^l |a_{k+1+t,l-t}| \\ 
= \sum_{k \geq 0 , l \geq 0} (l+1)|a_{k+1+t_{k,l},l-t_{k,l}}|\leq d^3||f||
\end{multline*}
\end{proof}
\begin{definition}
For $V=R^{\oplus n}$ we define $\widehat{S_{R}}(V) \cong \underset{m}\lim S_{R}(V)/I^m$ where $I$ is the ideal $(x_1, \dots, x_n) \subset R[x_1, \dots, x_n] \cong  S_{R}(V)$.
\end{definition}
It is an object of $\mathsf{Comm}(\Ind(\mathsf{Ban}_R))$ for any Banach ring $R$, and when $R$ is non-archimedean, it is also an object of $\mathsf{Comm}(\Ind(\mathsf{Ban}^{na}_R))$.
When $n=1$, $V=R$, this takes on a particularly simple form as $ \displaystyle\prod_{i=0}^{\infty} R$, which is nuclear and metrizable over $R$ because $R$ is and these properties are preserved by countable products  \cite{ben2020fr}. Therefore, functor $(-)\wotimes_{R}  \displaystyle\prod_{i=0}^{\infty} R$ therefore distributes over countable products \cite{ben2020fr}. In general, it easy to see that, $\widehat{S_{R}}(V\oplus W) \cong \widehat{S_{R}}(V) \wotimes_{R} \widehat{S_{R}}(W)$ and that in general, $\widehat{S_{R}}(V)$ is nuclear (hence flat) and metrizable over $R$ and that the completed tensor with it  distributes over countable products  \cite{ben2020fr}. For simplicity, we write $R[[x]]$ for $\widehat{S_{R}}(R)$. The product structure is given by the usual formulae for multiplying formal power series and the \(\mathsf{Ind}(\mathsf{Ban}_R)\) structure is determined by 
\begin{equation}\label{eqn:fps}
R[[x]] \cong  \displaystyle\prod_{i=0}^{\infty}  R \cong ``\underset{\psi \in \Psi(\mathbb{N})}\colim" \prod_{i \in \mathbb{N}}{}^{\leq 1} R_{\psi(i)^{-1}}.
\end{equation}
We similarly write $R[[x,y]] \cong R[[x]] \wotimes_{R} R[[y]]$ for $\widehat{S_{R}}(R \oplus R) \cong  \widehat{S_{R}}(R) \wotimes_{R} \widehat{S_{R}}(R)$ etc. and so 
\begin{equation}\label{eqn:twofps}R[[x,y]] \cong  \displaystyle\prod_{i,j=0}^{\infty} R  \cong ``\underset{\psi \in \Psi(\mathbb{N}^2)}\colim" \prod_{i,j \in \mathbb{N}}{}^{\leq 1} R_{\psi(i,j)^{-1}}.
\end{equation}
We have that 
\begin{equation}\label{eqn:decomForm}
R[[x_1, \dots, x_n]] \wotimes_{R} R[[y_1, \dots, y_m]]  \cong R[[x_1, \dots, x_n, y_1, \dots, y_m]]  
\end{equation}
and when $R$ is non-archimedean,
\begin{equation}\label{eqn:decomFormNA}
R[[x_1, \dots, x_n]] \wotimes^{na}_{R} R[[y_1, \dots, y_m]]  \cong R[[x_1, \dots, x_n, y_1, \dots, y_m]]  
.
\end{equation}

\begin{lemma}\label{lem:formal_poly_diag_strict}
Let \(R\) be a Banach ring and \(r>0\) an arbitrary positive real number. Then the sequence 
\[0 \to \widehat{S_{R}}(R) \wotimes_{R} \widehat{S_{R}}(R) \to  \widehat{S_{R}}(R) \wotimes_{R} \widehat{S_{R}}(R)  \to \widehat{S_{R}}(R) \to 0
\]
is strictly exact in the quasi-abelian category \(\mathsf{Ind}(\mathsf{Ban}_R)\). If $R$ is non-archimedean then the same is true in  \(\mathsf{Ind}(\mathsf{Ban}^{na}_R)\). 
\end{lemma}

\begin{proof}We will use the descriptions in equations (\ref{eqn:fps}) and (\ref{eqn:twofps}) together with the characterization in \ref{lem:filtered_strict}.  In order to see the strictness of the first map, consider any map $\psi: \mathbb{N}^2 \to \mathbb{Z}_{\geq 1}$ and any formal power series $f= \underset{(i,j) \in \mathbb{N}^2}\sum a_{i,j}y^iz^j$ which vanishes on the diagonal and satisfies \[
C=\underset{(i,j) \in \mathbb{N}^2}\sup \frac{|a_{i,j}|}{\psi(i,j)}< \infty.\] Then we can define a map $\phi: \mathbb{N}^2 \to \mathbb{Z}_{\geq 1}$ using the formula 
\[\phi(k,l)= (l+1) \eta(k,l) 
\]
where
\[\eta(k,l)= \underset{s=0, \dots, l}\prod \psi (k+1+s, l-s)
.\]
Then the power series $g =  \underset{(k,l) \in \mathbb{N}^2}\sum b_{k,l}y^kz^l \in R[[y,z]]$ defined by $f/(y-z)$ satisfies 
\[\underset{(k,l) \in \mathbb{N}^2}\sup \frac{|b_{k,l}|}{\phi(k,l)}= \underset{(k,l) \in \mathbb{N}^2}\sup \frac{|\underset{t=0, \dots, l}\sum a_{k+1+t,l-t}|}{\phi(k,l)}\leq \underset{(k,l) \in \mathbb{N}^2}\sup \underset{t=0, \dots, l}\sup \frac{ |a_{k+1+t,l-t}|}{\eta(k,l)}.
\]
Therefore 
\[\underset{(k,l) \in \mathbb{N}^2}\sup \frac{|b_{k,l}|}{\phi(k,l)} \leq \underset{(k,l) \in \mathbb{N}^2}\sup \underset{t=0, \dots, l}\sup \frac{ |a_{k+1+t,l-t}|}{\psi (k+1+t, l-t)} \leq C
\]
and this concludes the proof of strictness of the bounded multiplication by $(y-z)$ map on $R[[y,z]]$. It is trivial to see that the second map is a colimit of bounded maps with bounded sections.
\end{proof}
\begin{lemma}\label{lem:TateDiagStrict} Let \(R\) be a Banach ring. Then
for any $r>0$, the diagonal sequence
\[0 \to R \langle \frac{y}{r},\frac{z}{r}\rangle \to  R\langle \frac{y}{r},\frac{z}{r}\rangle \to R\langle\frac{x}{r}\rangle \to 0
\]
is a strictly exact  in  \(\mathsf{Ind}(\mathsf{Ban}_R)\). If $R$ is non-archimedean then the same is true in  \(\mathsf{Ind}(\mathsf{Ban}^{na}_R)\). 
\end{lemma}

\begin{proof}
Again, we only need to show that the map \((y-z)R\gen{\frac{y}{r},\frac{z}{r}} \to R\gen{\frac{y}{r},\frac{z}{r}}\), \(f \mapsto \frac{f(y,z) - f(z,z)}{y-z}\) is a bounded \(R\)-linear map.  Consider \(f = \underset{i,j \geq 0}\sum a_{i,j} x^i y^j\) with norm \(\norm{f} = \underset{i,j \geq 0}\sup \abs{a_{i,j}} r^{i+j}\). Using the estimates from Lemma \ref{lem:poly_diag_strict}, we see that \(\frac{f(y,z)-f(z,z)}{y-z}\) has norm given by

\begin{multline*}
\norm{\sum_{k,l \geq 0} c_{k,l} y^k z^l} = \underset{k,l \geq 0}\sup \abs{c_{k,l}} = \abs{c_{K,L}} = \abs{\sum_{t=0}^L a_{K+1+t, L+t}} \\
 \leq \sum_{t=0}^L \abs{a_{K+1+t, L- t}} \leq (L+1) \sup_{k,l} \abs{a_{k,l}} \leq (L+1) \norm{f},
\end{multline*}

\noindent where \(c_{k,l} = \sum_{t=0}^l a_{k+1+ t, l-t}\), and \((K,L)\) satisfy $|c_{K,L}|=\underset{k,l \geq 0}\sup \abs{c_{k,l}}$. The pair $(K,L)$ exists because the $\underset{k,l}\lim \abs{c_{k,l}}=0$  
\end{proof}
\begin{remark}\label{rem:wondered}The reader may have wondered why we have not stated a lemma similar to Lemma \ref{lem:TateDiagStrict} for the disk algebras $R \{\frac{x}{r}\}$. However the map 
\[\mathbb{Z} \{x, y\} \stackrel{x-y}\longrightarrow \mathbb{Z} \{x, y\}
\]
is not strict because 
\[(x-y)(x^{n-1}+  x^{n-2} y+ x^{n-3} y^2 + \cdots + x y^{n-2} + y^{n-1})= x^{n}-y^{n}
\]
meaning the norm can shrink by an arbitrarily high factor, in this case it shrinks from $n$ to $2$. As explained in Example \ref{ex:pad} the disk algebras over $\mathbb{Z}$ or even over the complex numbers, do not seem to participate in homotopy epimorphisms.
\end{remark}

Let $V=R_{r_1} \oplus \cdots \oplus R_{r_n}$.
Let \(S_R^{\leq 1}(V)= R \{ \frac{x_1}{r_1}, \dots,  \frac{x_n}{r_n}\}\) denote the contracting symmetric algebra - this is a free object in the category \(\mathsf{Comm}(\mathsf{Ban}^{\leq 1}_R)\). It is shown in \cite{Bambozzi-Ben-Bassat:Dagger}, that for every $\rho < r$ that the natural morphism $R \{ \frac{x_1}{r_1}, \dots,  \frac{x_n}{r_n}\}^\dagger \to R \{ \frac{x_1}{\rho_1}, \dots,  \frac{x_n}{\rho_n}\}^\dagger $ is a homotopy epimorphism for any Banach ring $R$. If $R$ is non-archimedean, then for every $\rho < r$, the natural morphism  $R \langle \frac{x_1}{r_1}, \dots,  \frac{x_n}{r_n}\rangle \to R \langle \frac{x_1}{\rho_1}, \dots,  \frac{x_n}{\rho_n}\rangle$ is a homotopy epimorphism.

\begin{lemma}\label{lem:colims}
Let \(R\) be any Banach ring, archimedean or not. Then we have
$
\underset{\rho>r}\colim R \{ \frac{x_1}{\rho_1}, \dots,  \frac{x_n}{\rho_n}\} \cong \underset{\rho>r}\colim R \langle \frac{x_1}{\rho_1}, \dots,  \frac{x_n}{\rho_n}\rangle
$ in  \(\mathsf{Ind}(\mathsf{Ban}_R)\). If \(R\) is non-archimedean the same can be said in \(\mathsf{Ind}(\mathsf{Ban}^{na}_R)\).
\end{lemma}

\begin{proof}
Using Lemma \ref{lem:compare} consider the bounded inclusions 
\[R \{ \frac{x_1}{\rho_1}, \dots,  \frac{x_n}{\rho_n}\} \to R \langle \frac{x_1}{\rho_1}, \dots,  \frac{x_n}{\rho_n}\rangle
\]
and for every $\rho>r$ define $\rho'=\frac{\rho+r}{2}$ so we get a bounded inclusion 
\[ R \langle \frac{x_1}{\rho_1}, \dots,  \frac{x_n}{\rho_n}\rangle \to R \{ \frac{x_1}{\rho'_1}, \dots,  \frac{x_n}{\rho'_n}\}.
\]
These both preserve the inclusion into the formal power series and their compositions are simply the usual maps in the ind-system.
\end{proof}

\begin{definition}\label{def:dagger}
Let \(R\) be a Banach ring and \(r \in \R_{\geq 0}^n\) a polyradius. We shall refer to the formal colimit from Lemma \ref{lem:colims} as the \textit{dagger algebra}

\[R \{\frac{x_1}{\varrho_1},\cdots, \frac{x_n}{\varrho_n}\}^\dagger \cong \underset{\varrho > r}\colim R \gen{\frac{x_1}{\varrho_1},\cdots, \frac{x_n}{\varrho_n}} \cong \underset{\varrho > r}\colim R \{\frac{x_1}{\varrho_1},\cdots, \frac{x_n}{\varrho_n}\}\] 
of \textit{overconvergent power series} on the polydisc of polyradius \(r\). We view this as an object of \(\mathsf{Ind}(\mathsf{Comm}(\mathsf{Ban}_R))\) or if $R$ is non-archimedean it can also be viewed as an object of  \(\mathsf{Ind}(\mathsf{Comm}(\mathsf{Ban}^{na}_R))\).
\end{definition}
It satisfies
\begin{equation}\label{eqn:decomDaggerNA}
R \gen{\frac{x_1}{r_1}, \dots, \frac{x_n}{r_n} }^{\dagger} \wotimes^{na}_{R} R \gen{\frac{y_1}{s_1}, \dots, \frac{y_n}{s_n} }^{\dagger} \cong R \gen{\frac{x_1}{r_1}, \dots, \frac{x_n}{r_n},  \frac{y_1}{s_1}, \dots, \frac{y_m}{s_m}}^{\dagger}.
\end{equation}
and
\begin{equation}\label{eqn:decomDagger}
R \{ \frac{x_1}{r_1}, \dots, \frac{x_n}{r_n} \}^{\dagger} \wotimes_{R} R \{ \frac{y_1}{s_1}, \dots, \frac{y_n}{s_n} \}^{\dagger} \cong R \{ \frac{x_1}{r_1}, \dots, \frac{x_n}{r_n},  \frac{y_1}{s_1}, \dots, \frac{y_m}{s_m}\}^{\dagger}.
\end{equation}
\begin{remark}\label{rem:dagger_equivalent}
In \cites{Cortinas-Cuntz-Meyer-Tamme:Nonarchimedean, Meyer-Mukherjee:Bornological_tf}, the authors define dagger algebras over non-archimedean discrete valuation rings in a slightly different way, using bornological analysis. Let \(\dvr\) be a complete discrete valuation ring with uniformiser \(\dvgen\). Suppose \(R\) is a (possibly non-commutative) bornologically torsion-free \(\dvr\)-algebra, then the \textit{linear growth bornology} is defined as the bornology generated by submodules of the form \(\sum_{n=0}^\infty \dvgen^n S^{n+1}\), where \(S \subseteq R\) is a bounded submodule. Denote \(R\) with the linear growth bornology by \(\ling{R}\). Completing in this bornology yields a complete, bornologically torsion-free \(\dvr\)-algebra \(R^\dagger \defeq \comb{\ling{R}}\). When \(R = \dvr[t_1,\dots, t_n]\) with the fine bornology, this specialises to the Monsky-Washnitzer algebra \[\dvr[t_1,\dots,t_n]^\dagger \cong \setgiven{\sum_{\alpha \in \N^n} c_\alpha t^\alpha}{\abs{c_\alpha} \leq C \beta^{\abs{\alpha}}, \text{ some } 0<\beta<1, C>0} \] by \cite{Cortinas-Cuntz-Meyer-Tamme:Nonarchimedean}*{Lemma 3.2.2}.  This has the bornology where a subset is bounded if it contains power series as above for a fixed \(C>0\) and \(0<\beta < 1\). Now suppose \(r = (1,\dots,1)\) in Definition \ref{def:dagger}. Then for each term \(\dvr \gen{\frac{t_1}{\varrho_1},\dots, \frac{t_n}{\varrho_n}} = \setgiven{\sum_{\alpha \in \N^n} d_\alpha x^\alpha}{ \abs{d_\alpha}\varrho^{\abs{\alpha}} \to 0} \) in the inductive limit, there is an obvious bounded inclusion into \(\dvr[t_1,\dots,t_n]^\dagger\): given such a $\sum_{\alpha \in \N^n} d_\alpha x^\alpha$ the relevant \(C\) is given by the upper bound of the sequence \((\abs{d_\alpha}\varrho^{\abs{\alpha}})_{\alpha \in \N^n}\) and \(\beta = \frac{1}{\varrho_1 \varrho_2 \cdots \varrho_n}\). Therefore, there is a unique bounded \(\dvr\)-algebra homomorphism \[\dvr \{t_1,\cdots,t_n\}^\dagger \to \dvr[t_1,\dots, t_n]^\dagger\] for the inductive limit bornology induced by the von Neumann bornologies on the Tate algebras on the left hand side.  

In the other direction, we use the universal property of dagger algebras \cite{Meyer-Mukherjee:Bornological_tf}*{Theorem 5.3}. More concretely, first note that there is a canonical bounded \(\dvr\)-algebra homomorphism \(\dvr[t_1,\dots,t_n] \to \dvr \gen{\frac{t_1}{\varrho_1},\dots,\frac{t_n}{\varrho_n}}\) into the Tate algebra. Since a Banach \(\dvr\)-algebra with the von Neumann bornology is a dagger algebra, there is a unique bounded \(\dvr\)-algebra homomorphism \[\dvr[t_1,\cdots, t_n]^\dagger \to \dvr \gen{\frac{t_1}{\varrho_1},\dots,\frac{t_n}{\varrho_n}} \subseteq \dvr \{t_1,\dots, t_n\}^\dagger,\] as required. 

%then taking colimit and equipping 
% the resulting algebra in \(\mathsf{CBorn}_\dvr\) with the inductive limit bornology with respect to the von Neumann bornologies on each term \(\dvr \gen{\frac{t_1}{\varrho_1},\dots, \frac{t_n}{\varrho_n}}\), we see that the identity map induces an isomorphism \[\dvr[t_1,\dots,t_n]^\dagger \cong \dvr\{ t_1, \dots, t_n\}^\dagger\] of complete bornological \(\dvr\)-algebras. 

\end{remark}

\begin{lemma}\label{lem:daggerdaigstrict}
Let \(R\) be any Banach ring, archimedean or not. Then
for any real number $r \geq 0$, the diagonal sequence
\[0 \to R\{\frac{y}{r},\frac{z}{r}\}^{\dagger} \to  R\{\frac{y}{r},\frac{z}{r}\}^{\dagger} \to R\{\frac{x}{r}\}^{\dagger} \to 0
\]
is a strictly exact in \(\mathsf{Ind}(\mathsf{Ban}_R)\). If \(R\) is non-archimedean the same can be said in \(\mathsf{Ind}(\mathsf{Ban}^{na}_R)\).

\end{lemma}

\begin{proof} 
By Lemma \ref{lem:filtered_strict} and \ref{lem:colims}, it suffices to show that for any \(\varrho > r\), the algebraically exact sequence \[0 \to R\gen{\frac{y}{\varrho}, \frac{z}{\varrho}} \to R\gen{\frac{y}{\varrho}, \frac{z}{\varrho}} \to R\gen{\frac{x}{\varrho}} \to 0\] is strict. The latter follows from Lemma \ref{lem:TateDiagStrict}. 
\end{proof}
Recall the definition of the analytic functions over $R$ on an open polydisk of polyradius $\bold{r}=(r_1, \dots, r_n)$ where $0<r_i\leq \infty$.
\[ \mathcal{O}_{R}(D_{\bold{r}})= \underset{\bold{p}<\bold{r}}\lim R\{\frac{x_1}{p_1}, \dots, \frac{x_n}{p_n} \}.
\]
When $r_i = \infty $ for $i=1, \dots, n$ then we write this as $\mathcal{O}^{an}_{R}(\mathbb{A}^{n})$.
\begin{remark}\label{rem:AnotherWay}
These open poly-disk algebras could also be written as a limit of $R\langle\frac{x}{\rho}\rangle$, or $R\langle\frac{x}{\rho}\rangle^{\dagger}=R\{\frac{x}{\rho}\}^{\dagger}$. They are objects of $\mathsf{Comm}(\Ind(\mathsf{Ban}_R))$ for any Banach ring $R$, and when $R$ is non-archimedean, also objects of $\mathsf{Comm}(\Ind(\mathsf{Ban}^{na}_R))$.
 \end{remark}
It was shown in \cite{Bambozzi-Ben-Bassat-Kremnizer:Stein} and \cite{ben2020fr} that for $\bold{t}=(r_1, \dots, r_n, s_1, \dots, s_m)$ that

\begin{equation}\label{eqn:decomStein}
 \mathcal{O}_{R}(D^{n}_{\bold{r}}) \wotimes_{R} \mathcal{O}_{R}(D^{m}_{\bold{s}}) \cong \mathcal{O}_{R}(D^{n+m}_{\bold{t}}).\end{equation}
When $R$ is non-archimedean we also have
\begin{equation}\label{eqn:decomSteinNA}
\mathcal{O}_{R}(D^{n}_{\bold{r}}) \wotimes^{na}_{R} \mathcal{O}_{R}(D^{m}_{\bold{s}}) \cong \mathcal{O}_{R}(D^{n+m}_{\bold{t}}).
\end{equation}
Furthermore, by \cite{ben2020fr}, $\mathcal{O}_{R}(D^{n}_{\bold{r}})$ is nuclear over $R$ and hence flat over $R$ and also metrizable over $R$. When $R$ is non-archimedean, the same holds in the non-archimedean category.

\begin{definition}\label{defn:Upsilon} Given a countable set $I$, let $\Upsilon(I)$ be the poset consisting of functions $\psi:I \to \mathbb{Z}_{\geq 1}$ with the order $\psi_{1}< \psi_{2}$ if $\psi_{1}(i)< \psi_{2}(i)$ for all $i \in I-J$ where $J$ is a finite subset of $I$. The category $\Upsilon(I)$ with objects $\underset{i \in I}\prod\mathbb{Z}_{\geq 1}$ can be thought of as categories of maps $I\to \mathbb{Z}_{\geq 1}$.
\end{definition}

The functions on the one dimensional open disk of radius $r$ will be denoted $\mathcal{O}_{R}(D^{1}_{r})$ and can be written \cite{ben2020fr} as 
\begin{equation}\label{eqn:Hol1var}``\underset{\psi \in \Upsilon(\mathbb{N})}\colim"\{f=\sum_{i=0}^{\infty} a_i x^{i} \in R[[x]] \ \ | \ \ \underset{n\in \mathbb{N}}\sup \left(\psi(n)^{-1} \sum_{i=0}^{\infty} |a_i|r_n^{i}\right)<\infty \}
\end{equation}
where $r_n \to r_-$ is strictly increasing.
For a multi-radius $(s,t)$, the two-dimensional open polydisk will be denoted $\mathcal{O}(D^{2}_{(s,t), R})$ and as shown in \cite{ben2020fr} can be written as
\begin{equation}\label{eqn:Hol2var}``\underset{\psi \in \Upsilon(\mathbb{N}^2)}\colim"\{f=\sum_{i,j=0}^{\infty} a_{i,j} x^{i}y^{j} \in R[[x,y]]  \ \ | \ \  \underset{n,m \in \mathbb{N}}\sup \left(\psi(n,m)^{-1} \sum_{i,j=0}^{\infty} |a_{i,j}|s_n^{i}t_m^{j}\right)<\infty \}
\end{equation}
where $s_n \to s_-$ and $t_m \to t_-$ are strictly increasing.

\begin{lemma}\label{lem:HoloStrict}
For any $r$, $0<r \leq \infty$ the diagonal sequence
\[0 \to \mathcal{O}_{R}(D^{2}_{(r,r)}) \to  \mathcal{O}_{R}(D^{2}_{(r,r)}) \to  \mathcal{O}_{R}(D^{1}_{r}) \to 0
\]
is strictly exact in \(\mathsf{Ind}(\mathsf{Ban}_R)\). If \(R\) is non-archimedean the same can be said in \(\mathsf{Ind}(\mathsf{Ban}^{na}_R)\).
\end{lemma}

\begin{proof} 
We will use the characterization in Lemma \ref{lem:filtered_strict}. Consider any map $\psi: \mathbb{N}^2 \to \mathbb{Z}_{\geq 1}$ and any formal power series $f= \underset{(i,j) \in \mathbb{N}^2}\sum a_{i,j}y^iz^j$ which vanishes on the diagonal and satisfies \[
C= \underset{n,m \in \mathbb{N}}\sup \left(\psi(m,n)^{-1} \sum_{i,j=0}^{\infty} |a_{i,j}|r_m^{i}r_n^{j}\right)< \infty\] 
and for all $m,n$
\[ C_{m,n}=\sum_{i,j=0}^{\infty} |a_{i,j}|r_m^{i}r_n^{j} <\infty
\]
where $r_n , r_m \to r_-$ are strictly increasing.
Consider the power series $g =  \underset{(k,l) \in \mathbb{N}^2}\sum b_{k,l}y^kz^l \in R[[y,z]]$ defined by $g=f/(y-z)$. Consider the composition of morphisms
\[R\{\frac{y}{r_m}, \frac{z}{r_n}\} \to R \langle\frac{y}{r_m}, \frac{z}{r_n}\rangle \longrightarrow  R\langle\frac{y}{r_m}, \frac{z}{r_n}\rangle \to R\{\frac{y}{r_{m-1}}, \frac{z}{r_{n-1}}\}
\]
where the middle one is given by \[f \mapsto \frac{f(y,z) - f(z,z)}{y-z}.
\]
They are all proven to be bounded in Lemmas \ref{lem:TateDiagStrict} and \ref{lem:compare}. Therefore there exist positive integers $D_{m,n}$ such that
 \[\sum_{k,l=0}^{\infty} |b_{k,l}|r_{m-1}^{k}r_{n-1}^{l} \leq D_{m,n} C_{m,n}.
 \]

Then we can define a map $\phi: \mathbb{N}^2 \to \mathbb{Z}_{\geq 1}$ using the formula 
\[\phi(m,n)= \psi(m+1,n+1)D_{m+1,n+1}
\]

Then the power series $g =  \underset{(k,l) \in \mathbb{N}^2}\sum b_{k,l}y^kz^l \in R[[y,z]]$ defined by $f/(y-z)$ satisfies 
\begin{equation}\begin{split}\underset{m,n \in \mathbb{N}}\sup \left( \frac{\sum_{k,l=0}^{\infty} |b_{k,l}|r_m^{k}r_n^{l} }{\phi(m,n)}\right) & = \underset{m,n \in \mathbb{N}}\sup \left( \frac{\sum_{k,l=0}^{\infty} |b_{k,l}|r_{m-1}^{k}r_{n-1}^{l} }{\phi(m-1,n-1)}\right)  \\
& \leq \underset{m,n \in \mathbb{N}}\sup \frac{ D_{m,n} C_{m,n}}{\phi(m-1,n-1)} \\
& = \underset{m,n \in \mathbb{N}}\sup \frac{  C_{m,n}}{\psi(m,n)} =C
\end{split}
\end{equation}
\end{proof}
\begin{remark}
As mentioned in Remark 
\ref{rem:AnotherWay} another way to prove Lemma \ref{lem:HoloStrict} would be to simply replace the summations over $i,j$ in the boundedness condition in Equations \ref{eqn:Hol1var} and \ref{eqn:Hol2var} with suprema.
\end{remark}

\subsection{The analytification functor}\label{subsec:analytification_functor}

In this subsection, we discuss the functoriality of the canonical maps from the polynomial \(R\)-algebra to any of the completions of the previous subsection. Let \(P_n \defeq R[x_1,\dots,x_n]\) be a finitely generated polynomial algebra, without any further (analytical) structure. That is, we view \(P_n\) as an object of the category \(\mathsf{Poly}_R\) of finitely generated polynomial \(R\)-algebras. As already seen at the start of this section, we can view \(P_n\) as an object of \(\mathsf{Comm}(\mathsf{Ind}(\mathsf{Ban}_R))\) by writing it as a filtered colimit 

\[ P_n \cong \bigcup_{m \in \N} \mathcal{F}_m(R[x_1,\dots,x_n]),\] where each \(\mathcal{F}_m(R[x_1,\dots,x_n])\) is a finite dimensional \(R\)-module, equipped with the norm of \(R\).  This is an object of \(\mathsf{Ind}(\mathsf{Ban}_R)\), and we obtain a functor 

\[ \mathsf{Poly}_R \to \mathsf{Comm}(\mathsf{Ind}(\mathsf{Ban}_R)),\] which we call the \textit{fine bornology functor}. It is trivially to see that the fine bornology functor is independent of the choice of coordinates \((x_1,\dots,x_n)\), and is thereby also well-defined. Next, consider the assignment which takes each polynomial algebra \(R[x_1,\dots,x_n] \in \mathsf{Comm}(\mathsf{Ind}(\mathsf{Ban}_R))\) with the fine structure to one of the $n$-dimensional analytic algebras $R[[x_1,\dots,x_n]]$, $\mathcal{O}_{R}(D^{n}_{\bold{p}})$ or $R \{\frac{x_1}{r_1}, \dots, \frac{x_n}{r_n} \}^\dagger$.  These are all objects of  \(\mathsf{Comm}(\mathsf{Ind}(\mathsf{Ban}_R))\).  If $R$ is non-archimedean we can use $R\gen{\frac{x_1}{t_1},\dots,\frac{x_n}{t_n}}$ in  \(\mathsf{Comm}(\mathsf{Ind}(\mathsf{Ban}^{na}_R))\).  Any algebraic map $R[x_1,\dots,x_n] \to R[x_1,\dots, x_m]$ is sent to the obvious map in \(\mathsf{Comm}(\mathsf{Ind}(\mathsf{Ban}_R))\). Again, it is easy to see that this assignment is functorial. We obtain another functor \[\mathsf{Poly}_R \to \mathsf{Comm}(\mathsf{Ind}(\mathsf{Ban}_R)),\] which plays the role of analytification of the affine \(n\)-space \(\mathbb{A}_{R}^n = \Spec(R[x_1,\dots,x_n])\).  

To extend the functor above to the category $SCR_R$ of simplicial commutative \(R\)-algebras, we first embed the discrete category \(\mathsf{Comm}(\mathsf{Ind}(\mathsf{Ban}_R))\) into the underlying \(\infty\)-category of the category \(\mathsf{sComm}(\mathsf{Ind}(\mathsf{Ban}_R))\) with the model structure defined in \cite{Ben-Bassat-Kremnizer:Nonarchimedean_analytic} and developed in \cite{Ben-Bassat-Kelly-Kremnizer:Perspective}. Since this category admits geometric realizations, we can take the Kan extension of the induced functor \(\mathsf{Poly}_R \to \mathsf{sComm}(\mathsf{Ind}(\mathsf{Ban}_R))= \mathsf{Comm}(\mathsf{sInd}(\mathsf{Ban}_R))\) along the inclusion \(\mathsf{Poly}_R \hookrightarrow SCR_R= \mathsf{sCAlg}_R\). More precisely, by a formalism developed by Lurie, the category of simplicial commutative algebras over $R$ is generated by sifted homotopy colimits by $\mathsf{Poly}_R$. The main fact which we rely on that if $\mathcal{S}$ is any $\infty$-category admitting sifted homotopy colimits then the inclusion $\mathsf{Poly}_R \to SCR_R$ induces an equivalence of categories 
\[\text{Funct}_{sft}(SCR_R, \mathcal{S}) \longrightarrow \text{Funct}(\mathsf{Poly}_R, \mathcal{S})
\]
where the left hand side is the full subcategory of $\text{Funct}(SCR_R, \mathcal{S})$ spanned by those commuting with sifted colimits. Denote by $\mathcal{F}$ the full subcategory of the $\infty$-category $\mathsf{Comm}(\mathsf{sInd}(\mathsf{Ban}_R))$ generated from the  $R[[x_1,\dots,x_n]]$ for $n=0, 1, \dots$  by sifted homotopy colimits. Denote by $\mathcal{H}$ the full subcategory of the $\infty$-category $\mathsf{Comm}(\mathsf{sInd}(\mathsf{Ban}_R))$ generated from the  $\mathcal{O}_{R}(D^{n}_{\bold{p}})$ for $n=0, 1, \dots$ for all possible $\bold{p}$ by sifted homotopy colimits.  Denote by $\mathcal{D}$ the full subcategory of the $\infty$-category $\mathsf{Comm}(\mathsf{sInd}(\mathsf{Ban}_R))$ generated from the $R \{\frac{x_1}{r_1}, \dots, \frac{x_n}{r_n} \}^\dagger$ for $n=0, 1, \dots$ for all possible $\bold{r}$ by sifted homotopy colimits. Denote by $\mathcal{T}$ the full subcategory of the $\infty$-category $\mathsf{Comm}(\mathsf{sInd}(\mathsf{Ban}^{na}_R))$ generated from the $R\gen{\frac{x_1}{t_1},\dots,\frac{x_n}{t_n}}$ for $n=0, 1, \dots$ for all possible $\bold{t}$ by sifted homotopy colimits. For $R$ non-archimedean, we get analytification functors $SCR_{R}\to \mathcal{F}$,  $SCR_{R}\to \mathcal{H}$,  $SCR_{R}\to \mathcal{D}$,  $SCR_{R}\to \mathcal{T}$. For $R$ archimedean, we get analytification functors $SCR_{R}\to \mathcal{F}$,  $SCR_{R}\to \mathcal{H}$,  $SCR_{R}\to \mathcal{D}$. The categories $\mathcal{F},  \mathcal{D},  \mathcal{T}, \mathcal{H}$ have obvious full subcategories of generators which play the same role inside them as the inclusion $\mathsf{Poly}_R \to SCR_R$.
\begin{example}
Given $n$ elements $f_i$ of $\mathbb{Z}[x_1, \dots, x_m]$, these can be used to define the following sifted homotopy colimit
$\mathbb{Z}[x_1, \dots, x_m] \wotimes^{\mathbb{L}}_{\mathbb{Z}[y_1, \dots, y_n]}\mathbb{Z}$ via $y_i \mapsto f_i$ and $y_i \mapsto 0$. This sifted homotopy colimit is the derived vanishing locus of the functions. One can similarly define analytic versions of this. They will be discussed more in subsection \ref{quotients}.
\end{example}

\section{Analytifications as homotopy epimorphisms}\label{sec:Analytification_hepi}

We now demonstrate a very general and explicit technique to prove that a given morphism \(f \colon A \to B\) in \(\mathsf{Comm}(\mathsf{Ind}(\mathsf{Ban}_R))\) is a homotopy epimorphism. This will be used to show that several different analytifications of the affine \(n\)-space all arise as homotopy epimorphisms. Our approach uses the \textit{Koszul resolution of the diagonal} associated to an object  \(C \in  \mathsf{Comm}(\mathsf{Ind}(\mathsf{Ban}_R))\), which we assume is flat over \(R\) in the sense of Definition \ref{def:flat}. The first thing we need in order to make sense of the boundary map in the usual Koszul resolution of a commutative algebra is the meaning of an `element' in the category \(\Ind(\mathsf{Ban}_R)\). 

\begin{definition}\label{def:element}
An \textit{element} \(x\) of an object \(V\) in a symmetric monoidal category $\fC$ with monoidal structure $\ootimes$ is a choice of morphism \(x \colon \id_{\fC} \to V\). We denote an element by \(x \in V\).  
\end{definition}

In concrete categories such as \(\mathsf{Mod}_R\) or \(\mathsf{CBorn}_R\), the above definition recovers the usual set-theoretic notion of an element. This is because in these categories, the functor \(\underline{\Hom}(R,-)\) maps an object to its underlying abelian groups. Furthermore, the functor \(\underline{\Hom}(R,-)\) is also faithful, so that a function is determined by its image on elements. On the other hand, this functor is not faithful on the category \(\Ind(\mathsf{Ban}_R)\). Consequently, there could exist non-zero inductive systems whose inductive limit is zero, so that such inductive systems have no non-trivial elements in the sense of Definition \ref{def:element}. 

%and \textcolor{red}{locally multiplicative}. Local multiplicativity ensures that \(C\) is a formal inductive limit of Banach \(R\)-algebras. 
%Then a strict monomorphism between locally multiplicative objects in \(\mathsf{Comm}(\mathsf{Ind}(\mathsf{Ban}_R)\) is an inductive limit of strict monomorphisms of Banach \(R\)-algebras. 
Let $\fC$ be a closed symmetric monoidal category with enough flat projectives, whose monoidal operation is denoted $\ootimes$.  Let $C$ be a commutative monoid in $\fC$.  Suppose that $c \in C$ is a non-trivial element.
%By flatness, we then have $R[y,z] \subset C\wotimes_{R}C \subset R[[y,z]]$. 
%Suppose also that multiplication by $(y-z)$ preserves $C\wotimes_{R}C$ and is strict as an endomorphism of $C\wotimes_{R}C$.  
Define the \textit{diagonal Koszul resolution} $K_{C} \to C$ of $C$ over $C\ootimes C$, where $K_C$ is defined by the complex of free modules over $C\ootimes C$:
\[C\ootimes C \longrightarrow C\ootimes C
\]
where the morphism is given by multiplication by $(c \otimes 1 -1 \otimes c)$. 

\begin{definition}\label{def:strictness_condition}
We say that the pair $(C, c)$ satisfies the \textit{strictness condition} if $C$ is flat over the unit for the monoidal structure, the morphism is a strict monomorphism and the (categorical) cokernel is isomorphic to $C$.  In this case, $K_C$ is strictly quasi-isomorphic to $C$.
\end{definition}

We learned about the following method from an article of Joseph L. Taylor \cite{taylor1972general}.
\begin{theorem}\label{thm:main_1}
Let $f: A\to B$ be any morphism  in \(\mathsf{Comm}(\fC)\) between objects $A$ and $B$. Suppose that the pairs $(A, a \in A)$ and $(B, b \in B)$ both satisfy the strictness condition of Defintion \ref{def:strictness_condition}, and that $f(a)=b$. Then $f$ is a homotopy epimorphism.
\end{theorem}

\begin{proof}
Since \((A,a)\) satisfies the strictness condition, the map \[K_A \to A\] is indeed a resolution, that is, an isomorphism in \(\fD(A \ootimes A)\). The morphism \(f \colon A \to B\) induces a commuting diagram
\[
\begin{tikzcd}
0 \arrow{r}{} & A \ootimes A \arrow{r}{a \otimes 1 - 1 \otimes a} \arrow{d}{f \otimes f} & A \ootimes A \arrow[r, two heads] \arrow{d}{f \otimes f} & A \arrow{d}{f}\\
0 \arrow{r}{} & B \ootimes B \arrow{r}{b \otimes 1 - 1 \otimes b} & B \ootimes B \arrow[r, two heads] & B, 
\end{tikzcd}
\]
where the bottom row is strictly exact by the hypothesis that \((B,f(a)=b)\) satisfies the strictness condition. Consequently, the map \(K_B \to B\) is an isomorphism in \(\fD(B \ootimes B)\). % We use Taylor's argument applied to the sequence in Lemma \ref{lem:daggerdaigstrict}. 
We have 
\[B\ootimes_{A}(A \ootimes A)\ootimes_{A} B \cong (B\ootimes_{A}A) \ootimes (A \ootimes_{A} B)  \cong B \ootimes B
\]
and therefore 
\[B\ootimes_{A}^{\mathbb{L}}B \cong   B\ootimes_{A}( K_A )\ootimes_{A}B \cong K_B \cong B.
\]
\end{proof}

\begin{corollary}\label{cor:formal_isocohomological}
For any Banach ring \(R\), the map \(R[x] \to R[[x]]\) is a homotopy epimorphism of commutative monoids in the category  \(\mathsf{Ind}(\mathsf{Ban}_R)\). If $R$ is non-archimedean, the same holds in the category  \(\mathsf{Ind}(\mathsf{Ban}^{na}_R)\).
\end{corollary}

\begin{proof}
Follows from \ref{lem:formal_poly_diag_strict}.
\end{proof}

\begin{corollary}\label{cor:polynomial_dagger_isocohomological}
Let $R$ be any Banach ring, archimedean or not. For any polyradii $0 \leq \rho < r$, the natural morphisms $R[x]\to R\{\frac{x}{r}\}^{\dagger} \to R\{\frac{x}{\rho}\}^{\dagger}\to R[[x]]$ are homotopy epimorphisms  of commutative monoids in the category  \(\mathsf{Ind}(\mathsf{Ban}_R)\). If $R$ is non-archimedean, the same holds in the category  \(\mathsf{Ind}(\mathsf{Ban}^{na}_R)\).
\end{corollary}

\begin{proof}
The strictness conditions hold by Lemmas \ref{lem:poly_diag_strict}, \ref{lem:daggerdaigstrict}, and \ref{lem:TateDiagStrict} finishing the proof. 

%For the second claim when $\rho>0$, we could use also Theorem \ref{thm:main_1} and \ref{lem:TateDiagStrict} to deduce that for 
%any pair of radii \(\rho < r\), the canonical inclusion \(R \langle \frac{x}{r} \rangle \to R \langle \frac{x}{\rho} \rangle\) is a homotopy epimorphism. 
%So by Proposition \ref{prop:tensor_isocohomological}, \(R \langle \frac{x_1}{r_1}, \cdots , \frac{x_n}{r_n} \rangle \to R \langle \frac{x_1}{\rho_1}, \cdots, \frac{x_n}{\rho_n} \rangle\) is a %homotopy epimorphism. Finally, by Proposition \ref{prop:filtered_colimit_isocohomological}, \(R\{\frac{x}{r}\}^{\dagger} \to R\{\frac{x}{\rho}\}^{\dagger}\) is a homotopy epimorphism.
\end{proof}

\begin{corollary}\label{cor:Tate_isocohomological}
Let $R$ be a non-archimedean Banach ring and $\rho< r$. Then the maps 
\[R[x] \to R\langle \frac{x}{r} \rangle \to R\langle \frac{x}{\rho} \rangle \to R[[x]]
\] are homotopy epimorphisms of commutative monoids in the category  \(\mathsf{Ind}(\mathsf{Ban}^{na}_R)\).
\end{corollary}

\begin{proof}
The strictness condition follows from Lemma \ref{lem:TateDiagStrict}.
\end{proof}

\begin{corollary}\label{cor:Stein_isocohomological}
For any radii \(r>p>0\), the maps \[R[x] \to \mathcal{O}_{R}(D^{1}_r) \to  \mathcal{O}_{R}(D^1_p) \to R[[x]]\] induced by canonical inclusions \(R[x] \to R\{\frac{x}{s}\}\) for \(s>r\) are homotopy epimorphisms of commutative monoids in the monoidal category  \(\mathsf{Ind}(\mathsf{Ban}_R)\). If $R$ is non-archimedean, the same holds in the category  \(\mathsf{Ind}(\mathsf{Ban}^{na}_R)\).
\end{corollary}

\begin{proof}
Follows from Lemma \ref{lem:HoloStrict}.
\end{proof}

\begin{theorem}\label{thm:polynomial_dagger_isocohomological2}
For any Banach ring \(R\) and any multiradius \(0 \leq r < \infty\), the canonical maps \[R[x_1,\dots,x_n] \to R \{\frac{x_1}{r_1}, \dots, \frac{x_n}{r_n} \}^\dagger \to  R[[x_1,\dots, x_n]]\] are homotopy epimorphism of commutative monoids in the monoidal category  \(\mathsf{Ind}(\mathsf{Ban}_R)\). Similarly, the maps, \[R[x_1, \dots, x_n] \to \mathcal{O}_{R}(D^{n}_{\bold{r}}) \to R[[x_1,\dots, x_n]]\] for  \(0 < r \leq \infty\), and \(R[x_1,\dots, x_n] \to R[[x_1,\dots, x_n]]\) are homotopy epimorphisms of commutative monoids in the monoidal category  \(\mathsf{Ind}(\mathsf{Ban}_R)\). For any multi-radii \(0 \leq t< r < p \leq \infty\) the maps \[\mathcal{O}_{R}(D^{n}_{\bold{p}}) \to R \{\frac{x_1}{r_1}, \dots, \frac{x_n}{r_n} \}^\dagger \to  R \{\frac{x_1}{t_1}, \dots, \frac{x_n}{t_n} \}^\dagger\] are homotopy epimorphisms of commutative monoids in the monoidal category  \(\mathsf{Ind}(\mathsf{Ban}_R)\). For any multi-radii \(0 < \bold{u}< \bold{p} < \bold{r} < \infty\), the maps 
\[R \{\frac{x_1}{r_1}, \dots, \frac{x_n}{r_n} \}^\dagger \to \mathcal{O}_{R}(D^{n}_{\bold{p}}) \to \mathcal{O}_{R}(D^{n}_{\bold{u}})\] are homotopy epimorphisms. If $R$ is non-archimedean, all of these are also  homotopy epimorphisms of commutative monoids in the monoidal category  \(\mathsf{Ind}(\mathsf{Ban}^{na}_R)\). Let $R$ be a non-archimedean Banach ring and $\rho< r$ a multi-radius. Then the maps 
\[R[x_1, \dots, x_n] \to R\langle \frac{x_1}{r_1}, \dots, \frac{x_n}{r_n} \rangle \to R\langle \frac{x_1}{\rho_1}, \dots, \frac{x_n}{\rho_n} \rangle \to R[[x_1, \dots, x_n]]
\] are homotopy epimorphisms of commutative monoids in the category  \(\mathsf{Ind}(\mathsf{Ban}^{na}_R)\).

\end{theorem}

\begin{proof}
By Corollaries \ref{cor:formal_isocohomological}, \ref{cor:polynomial_dagger_isocohomological}, \ref{cor:Tate_isocohomological}, and \ref{cor:Stein_isocohomological} we get the one-dimensional versions of these statements. To finish the proof we use Proposition \ref{prop:tensor_isocohomological_disc} applied to the decompositions \ref{eqn:decomPoly},  \ref{eqn:decomPolyNA}, \ref{eqn:decomTate}, \ref{eqn:decomForm}, \ref{eqn:decomFormNA}, \ref{eqn:decomDaggerNA}, \ref{eqn:decomDagger}, \ref{eqn:decomStein}, \ref{eqn:decomSteinNA}.\end{proof}
\begin{remark}\label{rem:hybrid}
By tensoring together with $\wotimes_{R}$ one dimensional objects like $R[x]$, $R[[x]]$, $R\{\frac{x}{r}\}^{\dagger}$ and $\mathcal{O}_{R}(D^{1}_r)$ and homotopy epimorphisms between them one gets all sorts of other hybrid analytifications $C, C'$ with homotopy epimorphisms 
$R[x_1, \dots, x_n] \to C \to C' \to R[[x_1, \dots, x_n]]$ in the general setting.  If $R$ is non-archimedean then by tensoring together $R[x]$, $R[[x]]$, $R\langle\frac{x}{r}\rangle$,  $R\{\frac{x}{r}\}^{\dagger}$ and $\mathcal{O}_{R}(D^{1}_r)$  and homotopy epimorphisms between them together with
$\wotimes^{na}_{R}$ one gets all sorts of other hybrid analytifications $C, C'$ with homotopy epimorphisms 
$R[x_1, \dots, x_n] \to C \to C' \to R[[x_1, \dots, x_n]]$ in the relavant category.
\end{remark}

%Using the results in \cite{Bambozzi-Ben-Bassat-Kremnizer:Stein} and \cite{Bambozzi-Ben-Bassat:Dagger}, we can prove that the canonical map \[S_R(V) \to S_R(V)^\dagger_{\leq 1}\] is an isocohomological embedding in the category \(\mathsf{Comm}(\mathsf{Ind}(\mathsf{Ban}_R))\).  

\begin{example}
As a concrete example, the result above says that for \(R = \Z_p\), \(V = \Z_p^n\), the inclusions \(\Z_p[x_1,\dots x_n] \to \Z_p \gen{x_1,\dots,x_n}\) and \(\Z_p[x_1,\dots,x_n] \to \Z_p\{x_1,\dots,x_n\}^\dagger\) are homotopy epimorphisms. We will revisit this example in Section \ref{sec:Hochschild}. 
\end{example}

\subsection{Quotients}\label{quotients}

In this subsection, we extend Theorem \ref{thm:main_1} to \textit{derived quotients} of polynomial algebras in \(n\)-variables. The quotient is a homotopy invariant notion that is useful in its own right, and will also be used to define derived localizations.
  %As an important application towards Hochschild homology, we are interested in proving the following:

%\begin{goal}
%The canonical map \(\Z_p[x_1,\ldots, x_n]/ I \to (\Z_p[x_1,\ldots, x_n]/I)^\dagger\) is a homotopy epimorphism, for a sufficiently nice ideal \(I = (f_1,\dots, f_k)\). 
%\end{goal}

\begin{definition}
 Suppose that $A$ is a simplicial commutative ring object in \(\mathsf{Ind}(\mathsf{Ban}_R)\), or a commutative ring object in complexes in \(\mathsf{Ind}(\mathsf{Ban}_R)\) in degrees less than or equal to zero and $a$ a degree zero element of $A$. Define 
 \[A\sslash a=A \wotimes^{\mathbb{L}}_{\mathbb{Z}[y]}\mathbb{Z},
 \]
  the maps being determined by $y \mapsto 0$ and $y \mapsto a$. Similarly, if $R$ and $A$ are non-archimedean and we chose to (we do not have to) work with the non-archimedean category, this quotient is defined by $A \wotimes^{na \mathbb{L}}_{\mathbb{Z}_{triv}[y]}\mathbb{Z}_{triv}$. If $A$ is a commutative ring object  in \(\mathsf{Ind}(\mathsf{Ban}_R)\), then by resolving $\mathbb{Z}$ by the complex $\mathbb{Z}[y] \stackrel{y}\longrightarrow \mathbb{Z}[y]$ we see that $A \sslash a$ can be thought of just the complex $A \stackrel{a}\longrightarrow A$ in degrees $-1$ and $0$ which is a commutative ring object in complexes of $R$-modules whose differential multiplies degree $-1$ elements by $a$. When this map is a strict monomorphism, then $A \sslash a=A/(a)$.  Define $A \sslash (a)=A \sslash a$ and by induction $A \sslash (a_1, \dots, a_n) = (A \sslash (a_1, \dots, a_{n-1})) \sslash a_n$. Equivalently, one can define this as $A \wotimes^{\mathbb{L}}_{\mathbb{Z}[y_1, \dots, y_n]}\mathbb{Z}$ and similarly in the non-archimedean context.
\end{definition}

In order to treat derived quotients and derived localizations in Section \ref{sec:derived_localisation}, we need to extend the notion of homotopy epimorphisms from algebras over a symmetric monoidal quasi-abelian category to commutative algebra objects in a sufficiently nice monoidal model category. More specifically, the higher categorical setup for the forthcoming work on derived analytic geometry \cite{Ben-Bassat-Kelly-Kremnizer:Perspective} is a \textit{homotopical algebraic context}, or HA-context, in the sense of \cite{toen2008homotopical}*{Section 1.1}.

\begin{definition}\label{def:derived_isocohomological}
Let \(\mathsf{M}\) be an HA-context whose monoidal structure is written as $\ootimes^{\mathbb{L}}$. A morphism of commutative monoids \(A \to B\) is called a \textit{homotopy epimorphism} if the canonical morphism \(B \ootimes_A^{\mathbb{L}} B \to B\) is a weak equivalence.
\end{definition}

The model categories of interest to us are the categories of complexes \(\mathsf{Ch}^{\leq 0}(\mathsf{Ind}(\mathsf{Ban}_R))\) in the setting where $R$ contains $\Q$, and simplicial objects \(\mathsf{s}(\mathsf{Ind}(\mathsf{Ban}_R))\), when \(R\) is an arbitrary Banach ring, with the projective model structure  \cite{Ben-Bassat-Kremnizer:Nonarchimedean_analytic}.  It easy to see that if \(A \to B\) is a homotopy epimorphism of ordinary commutative algebras in \(\mathsf{Ind}(\mathsf{Ban}_R)\), then it is a homotopy epimorphism in \(\mathsf{Comm}(\mathsf{s}(\mathsf{Ind}(\mathsf{Ban}_R))\), viewed as \(0\)-connective objects. 

\begin{proposition}\label{prop:special_quotients}
Suppose that $f:A\to B$ is a homotopy epimorphism of simplical commutative rings or commutative differential graded rings in \(\mathsf{Ind}(\mathsf{Ban}_R)\). Then for any $a_1, \dots, a_n \in A$, the natural morphism $A \sslash (a_1, \dots, a_n) \to B  \sslash (f(a_1), \dots, f(a_n))$ is a homotopy epimorphism. The analogous statement also holds if $R$ is non-archimedean and we work in \(\mathsf{Ind}(\mathsf{Ban}^{na}_R)\).
\end{proposition}

\begin{proof} We will do the general case as the non-archimedean situation is similar. It is easy to see using the definitions that for $a$ an element of $A$ that $(A \sslash a)\wotimes^{\mathbb{L}}_{A}B \cong B \sslash f(a)$. Indeed,
\[(A \sslash a)\wotimes^{\mathbb{L}}_{A}B = (\mathbb{Z}  \wotimes^{\mathbb{L}}_{\mathbb{Z}[y]}A)\wotimes^{\mathbb{L}}_{A}B \cong  \mathbb{Z}  \wotimes^{\mathbb{L}}_{\mathbb{Z}[y]}(A\wotimes^{\mathbb{L}}_{A}B) \cong  \mathbb{Z}  \wotimes^{\mathbb{L}}_{\mathbb{Z}[y]}B =  B \sslash f(a).
\]
Furthermore, the natural morphism $A \sslash a \to B \sslash f(a)$ is induced from $f$ by derived tensoring over $A \to A \sslash a$.   Now we just use the fact (Proposition \ref{prop:derived_base_change}) that derived tensoring a homotopy epimorphism along any morphism produces another homotopy epimorphism. The proof follows from a simple induction argument.
\end{proof}

\begin{remark} If $A$ is a commutative ring object  in \(\mathsf{Ind}(\mathsf{Ban}_R)\), then working in the category of commutative ring object in complexes in \(\mathsf{Ind}(\mathsf{Ban}_R)\) in degrees less than or equal to zero $A \sslash (a_1, \dots, a_{n})$ is simply the ordinary Koszul resolution
\[(A \stackrel{a_1}\longrightarrow A)\wotimes_{A} \cdots \wotimes_{A}(A \stackrel{a_n}\longrightarrow A).
\]
This follows from iterating the equivalence 
\[(A \wotimes^{\mathbb{L}}_{\mathbb{Z}[y]}\mathbb{Z})\wotimes^{\mathbb{L}}_{A}(A \wotimes^{\mathbb{L}}_{\mathbb{Z}[z]}\mathbb{Z}) \cong (A \wotimes^{\mathbb{L}}_{\mathbb{Z}[y]}\mathbb{Z})\wotimes^{\mathbb{L}}_{\mathbb{Z}[z]}\mathbb{Z}.\]
\end{remark}

\begin{example}\label{ex:pad}
Notice that the map $\mathbb{Z}\{x\} \to \mathbb{Z}\{2x\}$ is not a homotopy epimorphism. Indeed, we can resolve $\mathbb{Z}\{2x\}$ as $\mathbb{Z}\{x\}\{y\} \stackrel{y-2x}\longrightarrow \mathbb{Z}\{x\}\{y\}$ where the map is actually a strict monomorphism unlike in Remark \ref{rem:wondered}. However when we apply $(-) \wotimes_{\mathbb{Z}\{x\}} \mathbb{Z}\{2x\}$ to this complex we get a non-strict morphism similarly to Remark \ref{rem:wondered}. In a similar way $\mathbb{Z} \to \mathbb{Z}_{p}$ is not a homotopy epimorphism but these can be replaced with homotopy epimorphisms $\mathbb{Z}\{x\} \to \mathbb{Z}\{2x\}^{\dagger}$ and $\mathbb{Z} \to \mathbb{Z}^{\dagger}_{p}$. In fact, this is consistent with the results above on quotients as the quotient map of the homotopy epimorphism $\mathbb{Z}\{x\} \to \mathbb{Z}\{px\}^{\dagger}$ by $(x-p)$ is the map $\mathbb{Z} \to \mathbb{Z}^{\dagger}_{p}$ while the quotient map of  $\mathbb{Z}\{x\} \to \mathbb{Z}\{px\}$ by $(x-p)$ is the map $\mathbb{Z} \to \mathbb{Z}_{p}$ which is not a homotopy epimorphism. Multipilication by $x-p$ is actually a strict monomorphism both on $\mathbb{Z}\{px\}$ and $\mathbb{Z}^{\dagger}_{p}\{px\}$, and yet not on $\mathbb{Z}_{p}\{px\}$. For a discussion of \[\mathbb{Z}_{p}^{\dagger}=\mathbb{Z}\{px\}^{\dagger}/(x-p)=\mathbb{Z}\{px\}^{\dagger} \sslash (x-p)=\mathbb{Z}\{px\}^{\dagger}\wotimes^{\mathbb{L}}_{\mathbb{Z}\{y\}}\mathbb{Z}\] where the maps are given by $y \mapsto 0$ and $y \mapsto x-p$ see \cite{ben2020fr}.
\end{example}

\subsection{A version for underived quotients in the rational case}

While Proposition \ref{prop:special_quotients} deals with derived quotients, it is intractable in general to verify conditions under which we have \(A \sslash (a_1,\cdots,a_n) \cong A/(a_1,\cdots,a_n)\). In this subsection, we specialise to the setting of non-trivially valued Banach fields, which we denote by \(k\), to distinguish it from Banach rings without any further assumptions. Herein, we can use a slightly different approach to prove that analytifications of polynomial algebras are homotopy epimorphisms. Throughout this subsection, we work in the categories \(\mathsf{Born}_k\) and \(\mathsf{CBorn}_k\) of (complete) convex, separated bornological \(k\)-vector spaces. The inclusion functor \(\mathsf{CBorn}_k \to \mathsf{Born}_k\) has a left adjoint, namely, the \textit{completion functor} \(\mathsf{Born} \to \mathsf{CBorn}_k\). Furthermore, by Proposition \ref{prop:bornologies_ind_systems}, these two categories can be identified with the categories \(\mathsf{Ind}^m(\mathsf{Norm}_k^{1/2})\) and \(\mathsf{Ind}^m(\mathsf{Ban}_k)\) of inductive systems of semi-normed and Banach spaces with monomorphic structure maps, using the \textit{dissection functor} ( see \cite{bambozzi}*{Equation 2.1.3.1, Definition 2.1.4}).  We first recall some basic results:

\begin{lemma}\label{lem:fine_implies_complete}
Any \(k\)-vector space with the fine bornology is a complete bornological \(k\)-vector space. 
\end{lemma}

Any \(k\)-linear map \(f \colon M \to N\) between \(k\)-vector spaces is automatically bounded, when we equip \(M\) and \(N\) with the fine bornology. This gives us a functorial assignment from the category of \(k\)-vector spaces to the category of complete bornological \(k\)-vector spaces. We actually have more:

\begin{lemma}
The fine bornology functor is exact; that is, it maps an exact sequence of \(k\)-vector spaces to a strictly exact sequence of complete bornological \(k\)-vector spaces. 
\end{lemma}

\begin{proof}
A bornological embedding (resp. bornological quotient map) in the sense of \cite{Meyer:HLHA}*{Definition 1.74} is equivalent to an ordinary injective (resp. surjective) \(k\)-linear map between vector spaces. 
\end{proof}

\begin{lemma}\label{lem:tensor_completed}
Let \(M\) be a bornological \(k\)-vector space with the fine bornology, and let \(N\) be an arbitrary complete bornological \(k\)-vector space. Then \(M \otimes_k N \cong M \wotimes_k N\), where the left hand side denotes the incomplete bornological tensor product.
\end{lemma}

\begin{proof}
By definition, we can write \(M \cong \underset{i \in I}\colim M_i\), where \(M_i \cong k^{n_i}\) is a finite-dimensional \(k\)-vector subspace of $M$. Then \(M \wotimes_k N\) is the inductive limit of \(M_i \wotimes_k N_j \cong N_j^{n_i}\), where \(N \cong \underset{j \in J}\colim N_j\). Since the inductive limit functor commutes with the tensor product functor, 
\[M \otimes_k N \cong \underset{(i,j) \in I \times J} \colim M_i \otimes_k N_j \cong \underset{(i,j) \in I \times J}\colim N_j^{n_i} \cong M \wotimes_k N \]as required. \qedhere 
\end{proof}

Recall that in any quasi-abelian category \(\fC\), a chain complex \((C,d)\) over \(\fC\) is called \textit{exact} if the canonical diagram \[\ker(d) \into C \to \ker(d)\] is a strict monomorphism-strict epimorphism pair.  That is, the above diagram is strictly exact. Under suitable finiteness assumptions, which we make precise, we can often show that a linear surjection between bornological vector spaces is automatically a strict epimorphism. This will be used to show that if a chain complex of bornological vector spaces is algebraically exact, it is also strictly exact.

\begin{definition}\cite{Meyer:HLHA}*{Definition 1.158}
A bornological vector space \(V\) is said to be \textit{countably generated} if its bornology contains a cofinal sequence, which we can take to be increasing. This means that there is an increasing sequence of  bounded subsets \((B_n)_{n \in \N}\) such that for any bounded subset \(B \subseteq V\), there is an \(n\) such that \(B_n\) absorbs \(B\). 
\end{definition} 

%\begin{example}\label{ex:Frechet_countably_generated}
%Almost by definition, a Frechet space with the von Neumann bornology is a countably generated complete bornological vector space. Here we can take the unit balls \(B_n\) for each of the semi-norms as the required sequence of bounded subsets. Furthermore, all \(\mathsf{LF}\)-spaces and \(\mathsf{LB}\)-spaces with the inductive limit bornology with respect to the von Neumann bornology are countably generated as complete bornological vector spaces.  
%\end{example}
\noindent
The main source of countably generated bornological vector spaces is Banach spaces and countable nductive systems of Banach spaces. The latter are sometimes known as \textit{LB-spaces}. We will use the following version of Buchwalter's Theorem:

\begin{theorem}\cite{closedgraph}*{Theorem 4.9}\label{lem:Closed-graph}
Let \(f \colon A \to B\) be a linear surjection between complete bornological \(k\)-vector spaces. Assume that \(A\) has a countable generated bornology. Then \(f\) is a strict epimorphism.
\end{theorem}

\begin{proof}
Let \(M \subseteq B\) be a bounded Banach subspace of \(B\). We need to show that there is a bounded Banach disc \(S \subseteq A\) such that \(f(S) = M\). By assumption, there is a countable basis \((S_n)_{n \in \N}\) of Banach discs for the bornology on \(A\). Furthermore, it can be arranged that for each \(n\), \(S_n \subseteq S_{n+1}\). Since \(f\) is a surjection, we have \(\bigcup_{i=0}^\infty (\overline{f(S_i)} \cap M) = M\). By the Baire Category Theorem, there is a \(j\) such that \(\overline{f(S_j)} \cap M\) has an interior point. Since \(\overline{f(S_j)}\) is a vector subspace, it spans the entire space \(M\). Finally, again by surjectivity, the interior point of \(\overline{f(S_j)} = M\) must lie in \(f(S_j)\), which again implies that the latter spans \(M\), as required. 
\end{proof}

\begin{corollary}\label{cor:algebraic_exactness_bornological}
Let \((C,d)\) be a chain complex of complete bornological \(k\)-modules, where for each \(n\), \(C_n\) has a countably generated bornology. Suppose \((C,d)\) is algebraically exact, it is also (strictly) exact relative to the quasi-abelian category structure on \(\mathsf{CBorn}_k\).
\end{corollary}

\begin{proof}
The exactness of \(C\) means that for each \(n\), the canonical map \(C_{n+1} \to \ker(d_n)\) is a surjection. The hypothesis on \(C_n\) and Lemma \ref{lem:Closed-graph} implies the result. 
\end{proof}

This leads to the following result:

\begin{proposition}\label{prop:fine_implies_bornological}
Let $k$ be a non-trivially valued Banach field and \(A = k[x_1,\dots,x_n]/J\) be a finite-type commutative \(k\)-algebra, viewed as a complete bornological \(k\)-vector space with the fine bornology. Consider a morphism \(f \colon A \to B\) of commutative, complete bornological \(k\)-algebras.
\begin{itemize}
\item  Suppose \(B\) is a flat \(A\)-module for the algebraic tensor product \(\otimes_A\), and satisfies \(B \wotimes_A B \cong B\);
\item  the bornology of \(B\) is countably generated.
\end{itemize}
Then \(A \to B\) is a homotopy epimorphism in the category \(\mathsf{CBorn}_k\).  
\end{proposition}

\begin{proof}
We know that the bar complex \(P_\bullet(A) \defeq (P_n  = A^{\wotimes_k n + 1}, b_n)\) where 
\begin{multline*}
b_n(x_0 \otimes \cdots \otimes x_n) = \sum_{j=0}^{n-1} (-1)^j x_0 \otimes \cdots \otimes x_j \cdot x_{j+1} \otimes \cdots \otimes x_n \\
+ (-1)^n x_n x_0 \otimes x_1 \otimes \cdots \otimes x_{n-1},
\end{multline*} \(x_i \in A\), yields a strict resolution of \(A\) by projective \(A \wotimes_k A\)-modules. Notice that $B\wotimes^{\mathbb{L}}_{A} B$ is represented by $B\wotimes_{A}P_\bullet(A) \wotimes_{A}B$. Note that since \(A\) and the constituent terms of the bar resolution have the fine bornology, and since the fine bornology functor is fully exact, it makes no difference whether we use the (completed) bornological tensor product or the algebraic tensor product over \(k\). In order to make things clear, let $Q_\bullet(A)$ be the same complex as $P_\bullet(A)$ where $Q$ uses only algebraic tensor products.  Now since \(A \to B\) is flat for the algebraic tensor product over \(A\), we have that $B \otimes^{\mathbb{L}}_{A} B$ is computed by \(B \otimes_A Q_\bullet (A) \otimes_A B \). This is exact as a chain complex of \(k\)-vector spaces in all degrees except for $0$ where it has cohomology $B \otimes_{A} B$. Now since \(A\) and \(P_\bullet(A)\) have the fine bornology, we have 
\[B \wotimes_A P_n (A) \wotimes_A B \cong B \otimes_A Q_n(A) \otimes_A B \] for all \(n \geq 1\) by Lemma \ref{lem:tensor_completed}. This is a chain complex of complete bornological \(k\)-vector spaces that is still algebraically exact. Now since the bornology of \(B\) is countably generated, each term in the complex above has countably generated bornology. Corollary \ref{cor:algebraic_exactness_bornological} now implies that this complex is also (strictly) exact for the quasi-abelian structure on \(\mathsf{CBorn}_k\).  Therefore $B\wotimes_{A}P_\bullet(A) \wotimes_{A}B$ has cohomology only in degree zero and so $B\wotimes^{\mathbb{L}}_{A} B \cong B\wotimes_{A} B \cong B$.
\qedhere
\end{proof}

\begin{example}[Overconvergent non-archimedean geometry]\label{ex:overconvergent_isocohomological}
The hypotheses of Proposition \ref{prop:fine_implies_bornological} apply to the rationalized dagger completion (see \cite{Meyer-Mukherjee:Bornological_tf} and Remark \ref{rem:dagger_equivalent}) \(\underline{R}^\dagger \defeq \comb{\ling{R}} \wotimes_{\Z_p} \Q_p\) of a torsion-free, finite-type commutative \(\Z_p\)-algebra \(R\) with the fine bornology. More concretely, by \cite{Cortinas-Cuntz-Meyer-Tamme:Nonarchimedean}*{Lemma 4.14}, the map \(R \to R^\dagger\) is algebraically flat. Furthermore, a subset in the dagger completion \(R^\dagger\) is bounded if and only if it is contained in the \(\Z_p\)-submodule generated by \[S_{\ceil{C}} = \setgiven{\sum_{j=0}^\infty p^j w_j}{w_j \in M^{\kappa_j}, M \subseteq R \text{ finitely generated}, \kappa_j \leq \ceil{C}(j+1), C> 0}.\] These submodules generate the bornology on \(R^\dagger\). Therefore, the bornology on \(R^\dagger\) is countably generated. This does not change if we tensor with \(\Q_p\), so the bornology on \(\underline{R}^\dagger\) is still countably generated.  
\end{example}

\begin{example}[Complex analytic geometry]
Consider an open polydisk inside the complex analytification of the scheme associated to a finitely generated $\mathbb{C}$-algebra $A$. By Corollary A1.2.2 of \cite{Neeman:AlgAn}, the ring of analytic functions on this open polydisk is algebraically flat over $A$. Therefore, as a filtered colimit in \(\mathsf{CBorn}_\C\)  of functions on larger open polydisks, the overconvergent functions $B$ on the closure of a open polydisk are also flat over $A$. The overconvergent functions are also countably generated. If $A=\C[x_1,\dots,x_n]/(f_1,\dots,f_k)$ then for the appropriate closed polydisk, $B$ is given by $\C\{x_1,\dots,x_n\}^{\dagger}/(f_1,\dots,f_k)$, similarly for other multi-radii. The hypotheses of Proposition \ref{prop:fine_implies_bornological} apply to these cases. 
\end{example}

%\begin{example}[Complex analytic geometry]
%Let \(X = \Spec(\C[x_1,\dots,x_n]/(f_1,\dots,f_k))\) be a finite-type affine scheme over \(\C\). Its \textit{complex analytification} \(X^\mathrm{an}\) is a complex analytic space whose ring of functions is given by the Frechet algebra \[\mathcal{O}^{an}_{\C}(\mathbb{A}^{n})/(f_1,\dots,f_k),\]
%which has a countably generated bornology, by Example \ref{ex:Frechet_countably_generated}. Furthermore, by GAGA, the canonical homomorphism \[\C[x_1,\dots,x_n]/(f_1,\dots,f_k) \to \mathcal{O}^{an}_{\C}(\mathbb{A}^{n})/(f_1,\dots,f_k)\] is algebraically flat. So the hypotheses of Proposition \ref{prop:fine_implies_bornological} apply.
%\end{example}

\section{Localizations as Homotopy Epimorphisms}\label{sec:derived_localisation}

For each of the types of localization near $(0, \dots, 0)$ in affine $n$-space discussed in Theorem \ref{thm:polynomial_dagger_isocohomological2} and Remark \ref{rem:hybrid} we can introduce (motivated by terminology from rigid analytic geometry) Weierstrass and Lauent localizations representing functions near or far from the zero locus of functions in an arbitrary affine analytic scheme or a derived affine analytic scheme. In a similar way we also introduce rational localizations.  All of the derived localizations we introduce geometrically represent the inclusion of an open set in the relevant Grothendieck model topology: they represent homotopy monomorphisms of derived analytic schemes.

\begin{definition}\label{def:W} Fix a Banach ring $R$ and let $A$ be a unital strictly commutative simplicial ring or differential graded ring in  \(\mathsf{Ind}(\mathsf{Ban}_R)\). Fix an analytification $C$ with $R[x] \subset C \subset R[[x]]$. Write $A_C = A \wotimes^{\mathbb{L}}_{R}C$.  If $a$ is a degree zero element of $A$ and $C$ is a $1$-dimentional analytic ring over $R$ then the derived Weierstrass localization of $A$ determined by $C$ and $a$ is the map $A \to A_C \sslash (x-a)$. For elements $a_1, \dots, a_n$ we can either iterate this construction for different $C$s or simply use an analytification $R[x_1, \dots, x_n] \to C \to R[[x_1, \dots, x_n]]$ and the ideal $(x_1-a_1, \dots, x_n-a_n)$ resulting in $A \to A_C \sslash (x_1-a_1, \dots, x_n-a_n)$. These give the same result. The definitions for $R$ and $A$ non-archimedean work similarly in  \(\mathsf{Ind}(\mathsf{Ban}^{na}_R)\) with its monoidal structure.
\end{definition}
Geometrically one should imagine that $a$ gives a map from the spectrum of $A$ to $\mathbb{A}^{1}_{R}$ and we are pulling back the spectrum of $C$. Indeed 
\[A \wotimes^{\mathbb{L}}_{R[x]} C \cong (A \wotimes^{\mathbb{L}}_{R} C) \wotimes^{\mathbb{L}}_{\mathbb{Z}[y]} \mathbb{Z}=A_C \sslash (x-a)
\]
where $x\mapsto a$ and $y \mapsto x-a$, $y \mapsto 0$.
\begin{lemma}\label{lem:WS} Any derived Weierstrass localization is a homotopy epimorphism.
\end{lemma}
\begin{proof}The morphism $A \to A \wotimes^{\mathbb{L}}_{R[x]} C$ is a derived base change of the homotopy epimorphism $R[x] \to C$ along the morphism $R[x]\to A$ determined by $x\mapsto a$. For $R$ and $A$ non-archimedean this works similarly in \(\mathsf{Ind}(\mathsf{Ban}^{na}_R)\) with its monoidal structure.
\end{proof}
\begin{example}If $A=R=\mathbb{Z}$ and $C= \mathbb{Z}\{px\}^{\dagger}$ and $a=p$ then we get the homotopy epimorphism $\mathbb{Z} \mapsto \mathbb{Z}_{p}^{\dagger}$ from Example \ref{ex:pad}. If $R=k$ is a non-archimedean field and $A$ is an affinoid algebra over $k$ and we work in the non-archimedean context and $C=k\langle \frac{x}{r} \rangle$ then since multiplication by $x-a$ is strict so we get the map $A \to A\langle\frac{x}{r}\rangle/(x-a)=A\langle \frac{a}{r}\rangle$ a usual Weierstrass localization thereby reproving results from \cite{Ben-Bassat-Kremnizer:Nonarchimedean_analytic} that this is a homotopy epimorphism. 
%If $R=\mathbb{Z}$, $A=\mathbb{Z}\{x\}$ and $a=2x$ and $C=\mathbb{Z}[y]$, the derived Weierstrass localization $\mathbb{Z}\{x\}[2x]$ is the two term complex in degrees $0$ and $-1$ given by
%\[\mathbb{Z}\{x\}[y] \stackrel{y-2x} \longrightarrow  \mathbb{Z}\{x\}[y].
%\]
%It represents a derived closed disk of radius $\frac{1}{2}$ inside the closed disk over $\mathbb{Z}$ of radius $1$. The map is actually a strict monomorphism so we have $\mathbb{Z}\{x\}[2x] \cong \mathbb{Z}\{x\}[y]/(y-2x)$ and so 
%\[\mathbb{Z}\{x\} \longrightarrow \mathbb{Z}\{x\}[y]/(y-2x)
%\]
%is a homotopy epimirphism.
%Despite this when we apply the functor $(-) \wotimes^{\mathbb{L}}_{\mathbb{Z}[y]}\mathbb{Z}\{y\}$ we get the two term complex in degrees $0$ and $-1$ given by
%\[\mathbb{Z}\{x,y\} \stackrel{y-2x} \longrightarrow  \mathbb{Z}\{x,y\}.
%\]
%The categorical cokernel is just $\mathbb{Z}\{2x\}$ and the map is a monomorphism but it is not strict. Still, there does not seem to be any homotopy epimorphism $\mathbb{Z}\{x\} \to B$ such that $H^{0}(B)=\mathbb{Z}\{2x\}$. 
Some interesting examples in the non-archimedean context were given in \cite{BaK}.
\end{example}
An important example of derived Weierstrass localization is the derived adic-compltion of an ideal. 
\begin{definition}\label{def:DerAdCom}
Fix a Banach ring $R$ and let $A$ be a unital strictly commutative simplicial ring or differential graded ring in  \(\mathsf{Ind}(\mathsf{Ban}_R)\). Given an element $a\in A$ (a map $R$ to $A$) the analytic derived adic-completion $A_{\hat{a}}$ is defined by the derived Weierstrass localization $A\wotimes^{\mathbb{L}}_{R[y]}R[[y]]$ where the map is given by $y\mapsto a$, and similarly for a finitely generated ideal $J$ we define the derived $J$-adic completion $A_{\hat{J}}$ by derived Weirestrass localization.
\end{definition}
\begin{lemma}\label{lem:hyperbola} Given $C$ a $1$-dimensional analytic ring with
 \[R[y] \to C \to R[[y]]\] 
The natural map $R[x]\to C[x] \sslash (1-xy)= R[x]\wotimes_{R}C \sslash (1-xy)$ is a homotopy epimorphism.
\end{lemma}
\begin{proof}
Consider the commutative diagram of horizontal and vertical short exact sequences
\[
\xymatrix @+2.8pc {   C[x]\wotimes_{R[x]} C[x] \ar[r]^{(1,-x)^{T}}  \ar@{=}[d] & C[x]\wotimes_{R[x]} C[x]^{\oplus 2} \ar[r]^{(x,1)} \ar[d]^{\begin{pmatrix} -1 & -w \\ 1 & z \end{pmatrix}} & C[x]\wotimes_{R[x]} C[x]  \ar[d]^{z-w}   \\ C[x]\wotimes_{R[x]} C[x]  \ar[r]^{(wx-1,1-zx)^T}  & C[x]\wotimes_{R[x]} C[x]^{\oplus 2} \ar[d]^{(\Delta^*, \Delta^*)}  \ar[r]^{(1-zx,1-wx)} & C[x]\wotimes_{R[x]} C[x]   \ar[d]^{\Delta^*}  \\  & C[x]  \ar[r]^{1-yx}   &  C[x]   
}
\]
in which we have used the notation of matrices multiplying column vectors where the vector is placed on the right of the matrix and $\Delta^*(z)=y=\Delta^*(w)$ and we have written the left copy of $C$ in the variable $z$ in place of $y$ and in written the right copy of $C$ in the variable $w$ in place of $y$.  Keeping in mind the isomorphism $ C[x]\wotimes_{R[x]} C[x] \cong  C\wotimes_{R} C \wotimes_{R} R[x]$, the rightmost column is strict exact by assumption on $C$ and the flatness of $R[x]$ over $R$. 
%Now write $C[x]\wotimes_{R[x]} C[x]^{\oplus 2} $ as a formal colimit of Banach $R$-modules and let $(l,m)^{T} \in V$ where $V$ is one of the $R$-modules. 
Because $\det \begin{pmatrix} -1 & -w \\ 1 & z \end{pmatrix}=w-z$, is strict, the middle column is strict exact as well. Since the left column is also strict exact, we can conclude that the top row represents the cone of the natural morphism \[(C[x] \sslash (1-xy))\wotimes^{\mathbb{L}}_{R[x]}(C[x] \sslash (1-xy))\to C[x] \sslash (1-xy).\] However, on the top row, the right map is split on the right and the left map is split on the left and so the top row is strictly exact. Hence $R[x]\to C[x] \sslash (1-xy)$ is a homotopy epimorphism. For $R$ and $A$ non-archimedean this works similarly in \(\mathsf{Ind}(\mathsf{Ban}^{na}_R)\) with its monoidal structure.
\end{proof}

\begin{definition}\label{defn:Laurent}
Fix a Banach ring $R$ and let $A$ be a unital strictly commutative simplicial ring or differential graded ring in  \(\mathsf{Ind}(\mathsf{Ban}_R)\). If $a$ is a degree zero element of $A$ and $C$ is a $1$-dimensional analytic ring over $R$ and $A_C = A \wotimes^{\mathbb{L}}_{R}C$ then the derived annular localization of $A$ determined by $C$ and $a$ is the map \[A \to A \wotimes^{\mathbb{L}}_{R[x]}( C[x] \sslash (1-xy))= (A \wotimes^{\mathbb{L}}_{R} C) \sslash (1-ay).
\]
given by $x \mapsto a$. For elements $a_1, \dots, a_n$ we can either iterate this construction for different $C$s or simply use an analytification $R[y_1, \dots, y_n] \to C \to R[[y_1, \dots, y_n]]$ and the maps $x_i \mapsto a_i$. These give the same result. Define an derived annular localization as an iteration of one dimensional derived annular localizations where we can chose a different $C$ at each step and an element $a$ from the previous localization. Define a derived Lauent localization to be any finite iteration of derived Weierstrass and derived annular localizations (where at each step we can chose a new $C$ and a new element $a$ from the previous step).  For $R$ and $A$ non-archimedean this works similarly in \(\mathsf{Ind}(\mathsf{Ban}^{na}_R)\) with its monoidal structure.
\end{definition}
%From a geometric point of view this is like a pullback of a domain in $\mathbb{A}^{1,an}_{R}$ by the inverse of the function $a$ thought of as a map from the spectrum of $A$ to $\mathbb{A}^{1,an}_{R}$.
\begin{lemma}\label{lem:LauLoc} Any derived Laurent localization is a homotopy epimorphism.
\end{lemma}
\begin{proof}The morphism $A \to A \wotimes^{\mathbb{L}}_{R[x]} (C[x] \sslash (1-xy))$ is a derived base change of the homotopy epimorphism $R[x] \to C[x] \sslash (1-xy)$ from Lemma \ref{lem:hyperbola} along the morphism $R[x]\to A$ determined by $x\mapsto a$.  An application of Proposition \ref{prop:derived_base_change} finishes the proof. For $R$ and $A$ non-archimedean this works similarly in \(\mathsf{Ind}(\mathsf{Ban}^{na}_R)\) with its monoidal structure.

\end{proof}
\begin{example}If we work in the non-archimedean context and $R=k$ is a non-archimedean field and $A$ is an affinoid algebra and $a\in A$ and we take $C=k\langle ry \rangle$ then since $C[x]$ is flat over $R[x]$ the annular localization is the complex in degrees $-1$ and $0$ given by
\[[A\wotimes^{na}_{k[x]}k\langle ry \rangle[x] \stackrel{1-xy}\longrightarrow A\wotimes^{na}_{k[x]}k\langle ry\rangle[x] ]
\]
or
\[[A\wotimes^{na}_{k}k\langle ry \rangle \stackrel{1-ay}\longrightarrow A\wotimes^{na}_{k}k\langle ry \rangle]
\]
which gives the usual formula $A \to A\langle ry \rangle/(1-ay)=A\langle \frac{r}{a}\rangle$ discussed in \cite{Ben-Bassat-Kremnizer:Nonarchimedean_analytic}.
\end{example}

%\begin{definition}Fix a Banach ring $R$. Let $A$ be a unital strictly commutative simplicial ring or differential graded ring in  \(\mathsf{Ind}(\mathsf{Ban}_R)\). Write $A_C = A \wotimes^{\mathbb{L}}_{R}C$. 

\begin{definition}\label{defn:rational} Fix a Banach ring $R$ and let $A$ be a unital strictly commutative simplicial ring or differential graded ring in  \(\mathsf{Ind}(\mathsf{Ban}_R)\). For $i$ from $1$ to $n$, let $C_i$ be a $1$-dimensional analytic ring written in the variable $y_i$. Let $C_{(n)}=C_1\wotimes_R \cdots \wotimes_R C_n$. Let $g, f_1, \dots, f_n$ be degree zero elements of $A$ such that they generate the unit ideal in $H^{0}(A)$ or $\pi_0(A)$. The derived rational localization of $A$ with respect to this data is the map
\[A \longrightarrow A \wotimes_{R}C_{(n)} \sslash (gy_1-f_1, \dots, gy_n-f_n)
\]
where we have written $A \wotimes^{\mathbb{L}}_{R}C_{(n)}$ as $A \wotimes_{R}C_{(n)}$ as $C_{(n)}$ is flat over $R$ and we have chosen a specific model for $A$. For $R$ and $A$ non-archimedean this works similarly in \(\mathsf{Ind}(\mathsf{Ban}^{na}_R)\) with its monoidal structure.
\end{definition}
\begin{lemma} \label{lem:RatLoc} Any derived rational localization is a homotopy epimorphism.
\end{lemma}
\begin{proof}
First notice that $f_n, g$ generate the unit ideal for the connected components of \[A \wotimes_{R}C_{(n-1)} \sslash (gy_1-f_1, \dots, gy_{n-1}-f_{n-1}).\] 
This implies that it is enough to show that the maps 
\[A \longrightarrow B=A \wotimes_{R}C \sslash (gy-f)
\]
are homotopy epimorphisms when $f,g$ generate the unit ideal in  $H^{0}(A)$ or $\pi_0(A)$ and $C$ is a one-dimesional analytic ring over $R$ written in the variable $y$. Choose elements $a,b\in A$ such that $af+bg=1$ in $H^{0}(A)$ or $\pi_0(A)$. In the following diagram, we suppress the symbol $\mathbb{L}$ since $C$ is flat over $R$ by assumption and we have chosen a specific model representing $A$. Consider the following diagram in the derived category. 
\[
\xymatrix @+2.8pc { 
A\wotimes_{R}C\wotimes_{R}C \ar[r]^{(f,g)^{T}}  \ar@{=}[d] & A\wotimes_{R}C\wotimes_{R}C^{\oplus 2} \ar[r]^{(-g,f)} \ar[d]^{\begin{pmatrix} -1 & -w \\ 1 & z \end{pmatrix}} & A\wotimes_{R} C\wotimes_{R}C \ar[d]^{z-w} \\ 
A\wotimes_{R}C\wotimes_{R}C  \ar[r]^{(f-gw,gz-f)^T}  & A\wotimes_{R}C\wotimes_{R}C^{\oplus 2} \ar[d]^{(\Delta^*, \Delta^*)}  \ar[r]^{(gz-f,gw-f)} & A\wotimes_{R} C\wotimes_{R}C   \ar[d]^{\Delta^*}  \\ 
  & A\wotimes_{R}C\wotimes_{R}C  \ar[r]^{gy-f}  &  A\wotimes_{R}C  & 
}
\]
We have used the notation of matricies multiplying column vectors where the vector is placed on the right of the matrix and $\Delta^*(z)=y=\Delta^*(w)$ and we have written the left copy of $C$ in the variable $z$ in place of $y$ and in written the right copy of $C$ in the variable $w$ in place of $y$. The bottom row represents $B$ and the middle row represents $B\wotimes^{\mathbb{L}}_{A}B$ and the map between them is simply the natural product map. The right column is an exact triangle by assumption on $C$. Because $\det \begin{pmatrix} -1 & -w \\ 1 & z \end{pmatrix}=w-z$, is strict, the middle column is an exact triangle as well and the left column is obviously an exact triangle. The top row is split by the maps 
\[A\wotimes_{R}C\wotimes_{R}C \stackrel{(a,b)}\longleftarrow A\wotimes_{R}C\wotimes_{R}C^{\oplus 2}  \stackrel{(-b,a)^{T}}\longleftarrow A\wotimes_{R}C\wotimes_{R}C
\]
and so it is an exact triangle and hence we are done. For $R$ and $A$ non-archimedean this works similarly in \(\mathsf{Ind}(\mathsf{Ban}^{na}_R)\) with its monoidal structure.
\end{proof}

\section{Hochschild homology}\label{sec:Hochschild}

In this section, we apply the results of Section \ref{sec:Analytification_hepi} towards Hochschild homology computations. We first recall the version of Hochschild homology we need from \cite{toen2008homotopical}. Let $\mathsf{M}$ be a fixed HA-context with monoidal structure $\ootimes^{\mathbb{L}}$ and \(A\) a unital, commutative monoid in \(\mathsf{M}\). The Hochschild homology complex is defined as  
\[\mathbb{HH}(A) =A \ootimes_{A \ootimes^{\mathbb{L}} A}^{\mathbb{L}} A.
\]

We then have the following base change property:

\begin{proposition}\label{prop:HH_base-change}
Let \(A \to B\) be a homotopy epimorphism in \(\mathsf{Comm}(\mathsf{M})\). Then 
\[\mathbb{HH}(A)\ootimes_{A}^{\mathbb{L}}B \cong \mathbb{HH}(B).\]
\end{proposition}
\begin{proof}
Since \(A \to B\) is a homotopy epimorphism, by Proposition \ref{prop:tensor_isocohomological}, so is \(A \ootimes^{\mathbb{L}} A \to B \ootimes^{\mathbb{L}} B\). So by Lemma \ref{lem:base_change_1} and Proposition \ref{prop:tensor_isocohomological}, we have the following string of weak equivalences:
\begin{equation}\begin{split}
\mathbb{HH}(A)\ootimes_{A}^{\mathbb{L}}B & \cong \left(A \ootimes_{A \ootimes^{\mathbb{L}} A}^{\mathbb{L}} A \right) \ootimes_{\left(A \ootimes_{A}^{\mathbb{L}} A\right)}^{\mathbb{L}} \left(B \ootimes_{B}^{\mathbb{L}} B \right)  \\
& \cong \left(A \ootimes_{A \ootimes^{\mathbb{L}} A}^{\mathbb{L}} A \right) 
\ootimes_{\left(A \ootimes_{A}^{\mathbb{L}} A\right)}^{\mathbb{L}} \left(B \ootimes_{A}^{\mathbb{L}} B \right)  \\
& \cong  \left(A \ootimes_{ A}^{\mathbb{L}} B \right) \ootimes_{\left((A \ootimes^{\mathbb{L}} A) \ootimes_{A}^{\mathbb{L}} A\right)}^{\mathbb{L}} \left(A \ootimes_{A}^{\mathbb{L}} B \right)  \\
& \cong  B \ootimes_{A \ootimes^{\mathbb{L}} A}^{\mathbb{L}} B  \cong   B \ootimes_{B \ootimes^{\mathbb{L}} B}^{\mathbb{L}} B  \cong \mathbb{HH}(B),
\end{split}
\end{equation}
which completes the proof.
\end{proof}
\begin{remark} \label{rem:CotTV}In the same level of generality, To\"{e}n and Vezzosi demonstrated that if $\mathsf{M}$ is an HA context and \(A \to B\) be a homotopy epimorphism in \(\mathsf{Comm}(\mathsf{M})\) that $\mathbb{L}_{B} \cong \mathbb{L}_{A} \ootimes_{A}^{\mathbb{L}} B$.
\end{remark}
Geometrically, this can be interpreted as follows. Let \(X = \Spec(A)\) be a finite type derived affine scheme over a base ring \(R\) and suppose we have chosen a Banach ring structure on $R$. Then the geometrization of Proposition \ref{prop:HH_base-change} says that the (derived) fibered product \(X^{\mathrm{an}} \times_{X^{\mathrm{an}} \times X^{\mathrm{an}}}^{h} X^{\mathrm{an}}\) is equivalent to the base change with \(X^{\mathrm{an}}\) of the (derived) fibered product \(X \times_{X \times X}^{h} X\). Equivalently, the (derived)  \textit{loop space} \(\mathcal{L}(X^{\mathrm{an}})\) of \(X^{\mathrm{an}}\) is the base change with \(X^{\mathrm{an}}\) of the (derived) loop space \(\mathcal{L}(X)\).
Further properties such as transversality with respect to a composable pair of morphisms is recorded in \cite{toen2008homotopical}*{Proposition 1.2.2.2}. 

\begin{remark}
Following \cite{toen2008homotopical} again, one can define a relative version of (topological) Hochschild homology. We mention this for the sake of completeness. Given a morphism $A \to B$ of commutative monoids, $\mathbb{HH}(B/A):=A \ootimes^{\mathbb{L}}_{\mathbb{HH}(A)}\mathbb{HH}(B)$. If $B \to C$ are homotopy epimorphism, then for an arbitrary morphism $A \to B$, we have by Proposition \ref{prop:HH_base-change} that $\mathbb{HH}(B/A)\ootimes_{B}^{\mathbb{L}}C \cong \mathbb{HH}(C/A)$.
\end{remark}

The most basic monoidal model categories that we apply this to are the category of simplicial ind-Banach modules over $\mathbb{Z}$ together with $\wotimes^{\mathbb{L}}_{\mathbb{Z}}$ or non-archimedean  simplicial ind-Banach modules over $\mathbb{Z}_{triv}$ together with $\wotimes^{na \mathbb{L}}_{\mathbb{Z}_{triv}}$. But $\mathbb{Z}$ could be replaced by any simplicial ind-Banach ring and  $\mathbb{Z}_{triv}$ could be replaced by any non-archimedean simplicial ind-Banach ring provided we take the correct category of modules. These are HA contexts by work in \cite{Ben-Bassat-Kelly-Kremnizer:Perspective}. The more classical HA context is based on the category of simplicial abelian groups with the projective model structure. In the next subsection, we investigate some consequences of $\mathbb{HH}(A)\wotimes_{A}^{\mathbb{L}}B \cong \mathbb{HH}(B)$ working relative to the quasi-abelian category of Ind-Banach modules over $R$ and  $\mathbb{HH}(A)\wotimes_{A}^{na \mathbb{L}}B \cong \mathbb{HH}(B)$ when $R,A,B$ are non-archimedean.

\subsection{Hochschild homology computations in analytic geometry}

We now apply Proposition \ref{prop:HH_base-change} towards Hochschild homology computations in analytic geometry. In this subsection we show that the analytification homomorphisms \(A \to A^{\mathsf{an}}\) introduced previously in \ref{subsec:analytification_functor}, can be used to reduce Hochschild homology computations from \(A^{\mathsf{an}}\) to those of \(A\). In practice, the Hochschild homology of \(A\) relative to $R$ is simple to compute. We also discuss the Hochschild homology of localizations. We do not explicitly spell out the non-archimedean aspects of in this section. If the ring objects in any statement are non-arechimedean, the monoidal structure can be replaced with the non-archimidean one.

Let \(R\) be a Banach ring, \(r = (r_1,\dots,r_n)\) a poly-radius, and let \[A = S_R(R_{r_1} \oplus \cdots \oplus R_{r_n}) \cong S_R(R^n) \cong R[x_1,\dots, x_n].\]
This is carries a structure as an algebra object in \(\mathsf{Ind}(\mathsf{Ban}_R)\), as explained in Section \ref{subsec:symmetric_algebra}. When \(R = k\) is a non-trivially valued Banach field then the \(k\)-algebra with the fine bornology is a complete, bornological \(k\)-algebra.  We have discussed four types of $n$-dimensional analytic rings in Section \ref{subsec:analytification_functor}, namely:

\begin{enumerate}
\item The Tate-algebra \(R\gen{\frac{x_1}{r_1},\dots,\frac{x_n}{r_n}}\) when $R$ is non-archimedean and $0<r_i< \infty$;
\item The dagger algebra \(R\{\frac{x_1}{r_1},\dots,\frac{x_n}{r_n}\}^\dagger\) where $0\leq r_i< \infty$
\item The Stein algebra $\mathcal{O}_{R}(D^{n}_{\bold{r}})$ where $0 < r_i \leq \infty$
\item The formal power series $R[[x_1, \dots, x_n]]$.
\end{enumerate}

We denote a choice of one of these analytifications by \(A^\mathsf{an}\). More generally, one could use a hybrid  $n$-dimensional  analytic ring as in Remark \ref{rem:hybrid}. The disk algebra  \(R\{\frac{x_1}{r_1},\dots,\frac{x_n}{r_n}\}\) can also be considered an analytification but we do not use it here because we did not prove that $R[x_1,\dots, x_n] \to R\{\frac{x_1}{r_1},\dots,\frac{x_n}{r_n}\}$ is a homotopy epimorphism. By Theorem \ref{thm:polynomial_dagger_isocohomological2}, each of the above four analytifications $A \to A^\mathsf{an}$ are homotopy epimorphisms. Now let \(I=(a_1,\dots,a_m)\) be an ideal of \(A\).

\begin{theorem}\label{thm:HH_analytified} For any simplical algebra object in \(\mathsf{Ind}(\mathsf{Ban}_R)\),
we have a quasi-isomorphism \[\mathbb{HH}(A^{\mathsf{an}}\sslash (I\cdot A^{\mathsf{an}})) \cong \left( A^{\mathsf{an}} \sslash (I\cdot A^{\mathsf{an}} ) \right)\wotimes_{A \sslash I}^{\mathbb{L}} \mathbb{HH}(A \sslash I)\] For $R$ and $A$ non-archimedean this works similarly in \(\mathsf{Ind}(\mathsf{Ban}^{na}_R)\) with its monoidal structure.
\end{theorem}

\begin{proof}
 By Proposition \ref{prop:special_quotients}, the morphism \[A \sslash I \to A^{\mathsf{an}} \sslash (I\cdot A^{\mathsf{an}})\] is a homotopy epimorphism in the monoidal model category \(\mathsf{Ch}(\mathsf{Ind}(\mathsf{Ban}_R))\). The theorem then follows from Proposition \ref{prop:HH_base-change}.
\end{proof}

Let us look at some consequences of this result. 

\begin{example}\label{ex:abstract_HH_base-change}
Let \(A=P_n/J\) be a complete intersection ring, viewed as an object of \(\mathsf{Ind}(\mathsf{Ban}_R)\), where \(P_n = R[x_1,\dots,x_n]\). This means that the ideal \(J\) is generated by a strict regular sequence for the fine bornology. Consequently, consider the derived quotient $A$ as the degree-zero object \(P_n \sslash J \cong P_n/J \). Let $P_n^{\mathsf{an}}$ be one of the analytifications of $P_n$ from section \ref{subsec:analytification_functor}. We then have, using Koszul resolutions,  \[P_n^{\mathsf{an}} \sslash (J\cdot P_n^{\mathsf{an}}) \cong (P_n/J) \wotimes_{P_n}^{\mathbb{L}} P_n^{\mathrm{an}} \cong K(P_n,J) \wotimes_{P_n} P_n^{\mathrm{an}} \cong K(P_n^{\mathrm{an}},J).\]  Therefore, 
\[\mathbb{HH}(P_n^{\mathsf{an}} \sslash (J\cdot P_n^{\mathsf{an}})) \cong K(P_n^{\mathrm{an}},J) \wotimes_{A}^{\mathbb{L}} \mathbb{HH}(A).\]
The Hochschild homology complex of a complete intersection ring in the algebraic case (over any commutative ground ring \(R\)) has been computed in \cite{guccione-guccione}. Since the fine bornology functor is strictly exact, the same complex, now viewed as an object in \(\mathsf{Ind}(\mathsf{Ban}_R)\) computes \(\mathbb{HH}(A)\) in \(\mathsf{Ind}(\mathsf{Ban}_R)\). Notice that we do not obviously have that $ K(P_n^{\mathrm{an}},J)$ is strict exact so this issue requires more analysis.
\end{example}
We have described in \ref{thm:polynomial_dagger_isocohomological2} and \ref{rem:hybrid} many types of $n$-dimensional analytic rings. If we choose one of them for each $n$, $\mathcal{C}=\{C_{n}\}_{n\in \mathbb{N}}$ that gives an analytification functor 
\[F_{\mathcal{C}}: SCR_R \to \mathsf{sComm}(\mathsf{Ind}(\mathsf{Ban}_R))
\] described in subsection \ref{subsec:analytification_functor}.  In particular, we have the one given by taking the fine structure on each polynomial algebra, call it $F_{\text{fine}}$. Suppose we have a collection of morphisms of analytic rings $\{f_n: C_{n} \to C'_{n}\}_{n\in \mathbb{N}}$. This induces a natural transformation $F_{\mathcal{C}} \longrightarrow F_{\mathcal{C}'}$ of functors. 
\begin{lemma}\label{lemma:full}Suppose we take a collection  $\{f_n: C_{n} \to C'_{n}\}_{n\in \mathbb{N}}$ of homotopy epimorphisms of analytic rings over a Banach ring $R$. Then the induced natural transformation $F_{\mathcal{C}} \longrightarrow F_{\mathcal{C}'}$ of functors $SCR_R \to \mathsf{sComm}(\mathsf{Ind}(\mathsf{Ban}_R))$ satisfies that for every simplicial commutative ring $A$ that the map $F_{\mathcal{C}}(A) \longrightarrow F_{\mathcal{C}'}(A)$ is a homotopy epimorphism.
\end{lemma}
\begin{proof}If we write $A$ as a homotopy sifted colimit of polynomial algebras then the analytification functors send $A$ respectively to the homotopy sifted colimit of the same diagram of $C_n$ and to the homotopy sifted colimit of the same diagram of $C'_n$. The map between these colimits is a colimit of homotopy epimorphisms between the terms of the diagram and so a homotopy epimorphism.
\end{proof}
\begin{remark}The fact that homotopy pushouts of homotopy epimorphisms are homotopy epimorphisms was discussed already in Proposition \ref{prop:tensor_isocohomological}.
\end{remark}
\begin{corollary}The case that most closely resembles standard notions of analytification such as appear in GAGA types theorems are the case of open disks with infinite multi-radii where $C_n=\mathcal{O}_{R}(\mathbb{A}^{n}_{R})$ for $A\in SCR_R$. We have homotopy epimorphisms $F_{\text{fine}}(A) \to F_{\mathcal{C}}(A)$ in this case. In general, the map from the fine version of a simplical commutative ring, to any other of its analytifications is a homotopy epimorphism. 
\end{corollary}
\begin{example}
Suppose \(A\) is an object in \(\mathsf{Comm}(Mod(R)) \subset SCR_R \), where \(R\) is a commutative ring of characteristic zero. Given a Banach structure on $R$ we can consider a choice of analytification of the all the polynomial algebras $P_n$ over $R$ as in Section \ref{subsec:analytification_functor}. Since the assignment \(P_n \to P_n^{\mathrm{an}}\) is a homotopy epimorphism, where \(P_\bullet \to A\) is a simplicial resolution of \(A\) by finite type polynomial algebras, the same holds by Lemma \ref{lemma:full} when we pass to $A \to A^{an}$ where $A^{an}$ is defined by \( P_\bullet^{\mathrm{an}}\). Now by an analytic version of the Hochschild-Kostant-Rosenberg Theorem proven in \cite{kkm-analytic_KHR}, together with the fact (see Remark \ref{rem:CotTV} )that $\mathbb{L}_A \wotimes_{A}^{\mathbb{L}}A^{\mathrm{an}} \cong \mathbb{L}_A^{\mathrm{an}}$ we have the following very general computation
\begin{equation}\begin{split}
\mathbb{HH}(A^{\mathrm{an}}) & \cong Sym_{A^{an}}(\mathbb{L}_{A^{\mathrm{an}}}[1]) \cong Sym_{A^{\mathrm{an}}}(A^{\mathrm{an}} \wotimes_A^{\mathbb{L}} \mathbb{L}_A[1]) \\ & \cong A^{\mathrm{an}}  \wotimes_A^{\mathbb{L}} Sym_A(\mathbb{L}_A[1])  \cong A^{\mathrm{an}}  \wotimes_A^{\mathbb{L}} \mathbb{HH}(A),
\end{split}
\end{equation}
where \(\mathbb{L}\) is the cotangent complex and it can easily be shown that the HKR isomorphisms produce the same base change isomorphism on Hochschild homology which have discussed.
\end{example}

\begin{example} Let $R$ be a Banach ring and $A$ a simplical algebra object in \(\mathsf{Ind}(\mathsf{Ban}_R)\). By Lemmas \ref{lem:WS}, \ref{lem:LauLoc} and \ref{lem:RatLoc} combined with Proposition \ref{prop:HH_base-change} we know that for any type of localization (Weierstrass, Laurent, or rational) $LA$ of $A$ that  \[\mathbb{HH}(A)\wotimes_{A}^{\mathbb{L}}LA \cong \mathbb{HH}(LA),\] similarly in the non-archimedean context.
If we iterate this or use the analytifications \(R[x_1,\dots,x_n] \to R[[x_1,\dots,x_n]]\), with $R$ trivially valued, we can recover results of Hubl \cite{HublBook}. For example, similarly to  
\ref{def:DerAdCom}, considering a derived algebraic adic (Weierstrass) localization we have a homotopy epimorphism $A \to A_{\hat{a}}$ for any simplicial algebra $A$ over $R$ given the fine Ind-Ban structure and an element $a\in A$. Therefore $ \mathbb{HH}(A_{\hat{a}}) \cong  \mathbb{HH}(A)\wotimes^{\mathbb{L}}_{A}A_{\hat{a}}$. When $\mathbb{HH}(A)$ is perfect over $A$, this gives $ \mathbb{HH}(A_{\hat{a}}) \cong  \mathbb{HH}(A)\otimes_{A}A_{\hat{a}}$. For $A$ a finitely generated commutative algebra over a field of characteristic zero and $J$ an ideal of $A$ and $\overline{P_J}$ denotes the $J$-adic completion then again using the derived algebraic adic (Weierstrass) localization we recover the results $HH_{*}(\overline{A_J}) \cong \overline{A_J}\otimes_{A} HH_{*}(A)$ of Proposition 4.2.1 of \cite{Cortinas-Cuntz-Meyer-Tamme:Nonarchimedean}.
\end{example}

\begin{example}\label{ex:bornological_HH_base-change}
For a more concrete example, we can use the alternative approach of Proposition \ref{prop:fine_implies_bornological}. This works over a non-trivially valued Banach field \(k\). Let \(A\) and \(B\) be as in Proposition \ref{prop:fine_implies_bornological}, so that \(A \to B\) is a homotopy epimorphism in \(\mathsf{Comm}(\mathsf{CBorn}_k)\). Viewing this as a homotopy epimorphism in the simplicial model category \(\mathsf{Comm}(\mathsf{sCBorn}_k)\), we use Proposition \ref{prop:HH_base-change} to directly obtain a quasi-isomorphism  
\[\mathbb{HH}(B) \cong B \wotimes^{\mathbb{L}}_A \mathbb{HH}(A).\]
\end{example}
\noindent
Example \ref{ex:bornological_HH_base-change} and Example \ref{ex:overconvergent_isocohomological} allows us to recover the following result in \cite{Cortinas-Cuntz-Meyer-Tamme:Nonarchimedean}:

\begin{corollary}[\cite{Cortinas-Cuntz-Meyer-Tamme:Nonarchimedean}*{Proposition 4.17}]\label{cor:CCMT_analytification}
Let \(A\) be a finite-type, torsion-free commutative algebra over a discrete valuation ring \(\dvr\) with uniformiser \(\dvgen\), viewed as a complete bornological algebra with the fine bornology. Let \(A^\dagger\) be its dagger completion, as in Remark \ref{rem:dagger_equivalent}. We then have 
\[\mathbb{HH}(\underline{A}^\dagger) \cong \underline{A}^\dagger \otimes_{\underline{A}} \mathbb{HH}(\underline{A}),\] where \(\underline{(-)} \defeq (-) \wotimes_\dvr \dvf\), with \(\dvf = \dvr[\dvgen^{-1}]\) being the fraction field of \(\dvr\).
\end{corollary}

\begin{remark}\label{rem:rigid_cohomology}
The Hochschild homology base change result in Corollary \ref{cor:CCMT_analytification} was used in \cite{Cortinas-Cuntz-Meyer-Tamme:Nonarchimedean} to compare the \textit{periodic cyclic homology} of \(\underline{A}^\dagger\) with the \textit{rigid cohomology} of \(A/\dvgen A\) with coefficients in \(\dvf\), in the case where \(\underline{A}\) is smooth over the fraction field \(\dvf\). To see this, we first note that for the smooth algebra \(\underline{A}\), the Hochschild-Kostant-Rosenberg Theorem implies an isomorphism \[\HH_*(\underline{A}) \cong \Omega_{\underline{A}/\dvf}^*\] in homology where the right hand side is locally free of finite rank. Therefore, by Corollary \ref{cor:CCMT_analytification}, we have \(\HH_*(\underline{A}^\dagger) \cong \underline{A}^\dagger \otimes_{\underline{A}} \Omega_{\underline{A}/\dvf}^*\).  This can be promoted to an isomorphism of mixed complexes \[(\mathbb{HH}(\underline{A}^\dagger), b, B) \cong (\underline{A}^\dagger \otimes_{\underline{A}} \Omega_{\underline{A}/\dvf}^*, 0 , d),\] where \(B\) is Connes's differential. Finally, this implies an isomorphism 
\[\HP_*(\underline{A}^\dagger) \cong \bigoplus_{j \in \Z} H_{\mathrm{rig}}^{* + 2j}(A/\dvgen A, \dvf),  \ \ \ \ * = 0,1.\]  
\end{remark}

\end{document}